\documentclass[a4paper,10pt]{article}
\usepackage[utf8x]{inputenc}
\usepackage{natbib}
\usepackage{amssymb,amsmath,amsfonts,amsxtra,amsthm}
\usepackage[colorinlistoftodos]{todonotes}
\usepackage{graphics,graphicx}
\usepackage{caption}
\usepackage{subcaption}
\usepackage{enumerate}   
\usepackage{float}
\usepackage[lined,boxed,commentsnumbered, ruled,vlined,linesnumbered,onelanguage]{algorithm2e}
\usepackage{bbm}
\usepackage[a4paper,top=3cm,bottom=2cm,left=3cm,right=3cm,marginparwidth=1.75cm]{geometry}
\usepackage[T1]{fontenc}
\usepackage{stmaryrd}

\newtheorem{Theorem}{Theorem}
\newtheorem{Conditions}{Conditions}
\newtheorem{Corollary}{Corollary}[Theorem]
\newtheorem{Lemma}[Theorem]{Lemma}

\title{The diameter of the minimum spanning tree of the complete graph with inhomogeneous random weights}
\author{Othmane SAFSAFI \\
       \\}

\date{--}
\begin{document}
\begin{titlepage}
    \centering
    \vfill
    {\bfseries\Large
        The diameter of the minimum spanning tree of the complete graph with inhomogeneous random weights\\
		\textit{}
 		
        \vskip2cm
		 Othmane SAFSAFI 
	}

    \vskip2cm

\begin{abstract}
We study a new type of random minimum spanning trees. It is built on the complete graph where each vertex is given a weight, which is a positive real number. Then, each edge is given a capacity which is a random variable that only depends on the product of the weights of its endpoints.  We then study the minimum spanning tree corresponding to the edge capacities. Under a condition of finite moments on the node weights, we show that  the expected diameter and typical distances of this minimum spanning tree are of order $n^{1/3}$. This is a generalization of the results of  \citet*{LNB09}. We then use our result to answer a conjecture in statistical physics about typical distances on a closely related object. This work also sets the ground for proving the existence of a non-trivial scaling limit of this spanning tree (a generalization of the result in  \citet*{LNCG13}). Our proof is based on a detailed study of rank-1 critical inhomogeneous random graphs, done in  \citet*{O20}, and novel couplings between exploration trees related to those graphs and Galton-Watson trees. 
\begin{figure*}[!htbp]
\centering
\includegraphics[width=1.\textwidth]{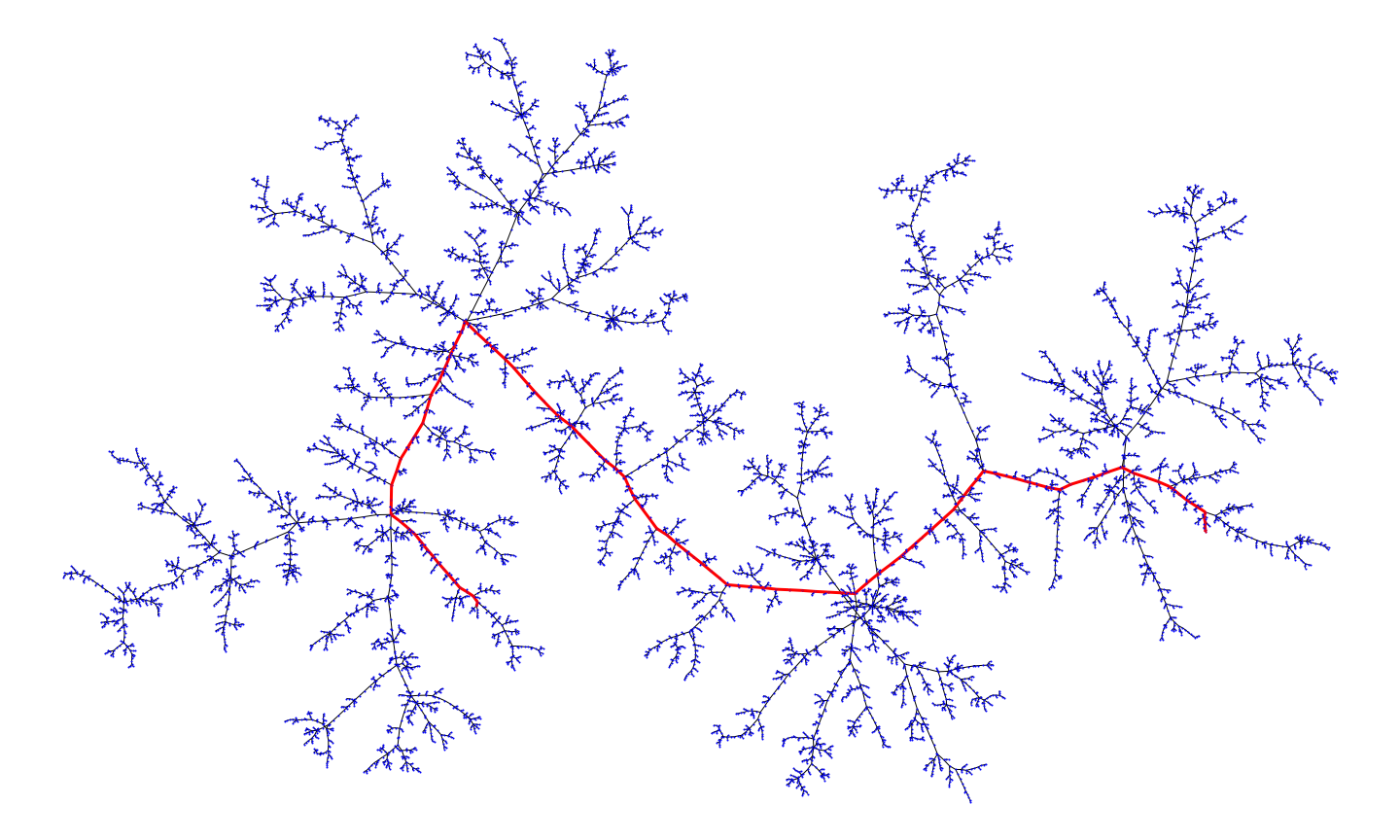}
\caption{An example of the random trees we study in this article. We highlighted a typical random path in the tree in red.}
\end{figure*}
\end{abstract}
\end{titlepage}
\section{Introduction}
\subsection{The model}
Let $G = (V,E)$ be a connected graph with $n$ nodes and $m$ edges. Let $e_1,e_2,...,e_{m}$ be positive real numbers that represent capacities on the edges of $G$. Assume that these capacities are all distinct, then there  exists a  unique spanning tree $T = (V,E_T)$ that minimizes the sum of the capacities on the edges:
$$ \sum_{e_i \in E_T} e_i.$$
We call this tree the minimum spanning tree (MST). If the graph is not connected the set of minimum spanning trees of its connected components is called the minimum spanning forest. If the weights are independent random variables with atomless distributions, the MST will be almost surely unique. 
The minimum spanning tree is an important object in combinatorial  optimization, it can be easily computed (see \citet*{KR56} and \citet*{P57}), and it can be used to construct approximations to more difficult problems such as the traveling salesman (\citet*{V01}). The study of random minimum spanning trees is also of independent interest in statistical physics (\citet*{CHE06}, \citet*{WU06}, \citet*{BR07}).    \par
From a probabilistic perspective, the minimum spanning tree has been studied extensively and on various graphs and models of randomness.
Let $\mathcal{T}_n$ denote the minimum spanning tree of a complete graph of size $n$ with i.i.d. $[0,1]$-uniform random capacities on its edges. \citet*{F85} proved that the total capacity of $\mathcal{T}_n$ converges to $\zeta(3) = \sum_{k=1}^{+\infty}\frac{1}{k^3}$. \citet*{aldous1990random} found the limit distribution of the degree of node $1$ in $\mathcal{T}_n$. These results are based on a local study of the minimum spanning tree. \par
Another type of questions concerns global properties of the random MST. For instance, \citet*{LNB09} showed that the diameter (maximum number of edges of a shortest path between two nodes) of $\mathcal{T}_n$ is of order $n^{1/3}$. This estimate was used crucially by \citet*{LNCG13} to show that there exists a measured metric space which is a scaling limit of $\mathcal{T}_n$ with distances scaled by $n^{-1/3}$.
In this paper, we extend the results of  \cite{LNB09} to a more general class of random minimum spanning trees. Let $n \in \mathbb{N}$ and $\textbf{W}=(w_1,w_2,..w_n)$, with $0< w_n \leq w_{n-1} \leq ... \leq w_1$, be a vector of positive weights, and consider the complete graph of size $n$. $\textbf{W}$ will always depend implicitly on $n$.
To each  non-oriented edge $\{i,j\}$, $i \neq j$,  we associate the random capacity $E_{\{i,j\}}$, which is an exponential random variable of rate $w_iw_j$.  The capacities are then used to create a sequence of graphs. For each $p \in [0,+\infty]$ let $G(\textbf{W},p)$ be the graph on $\{1,2....,n\}$ containing the edges of capacity at most $p$, so the edge set of $G(\textbf{W},p)$ is:
$$\left\{ \{i,j\} |  E_{\{i,j\}} \leq p \right\}.$$ 
Then  $(G(\textbf{W},p))_{p \in [0,+\infty]}$ is an increasing sequence of graphs (for inclusion), and for each fixed value of $p$, $G(\textbf{W},p)$ is called a rank-1 inhomogeneous random graph.
Now for each $p$, consider the forest $\mathcal{T}(\textbf{W},p)$ constructed  by deleting edges from $G(\textbf{W},p)$ as follows: \par

We construct a sequence of graphs  $(\mathcal{G}(\textbf{W},p,i))_{\binom{n}{2} \geq i \geq 1}$ such that $\mathcal{G}(\textbf{W},p,\binom{n}{2}) = \mathcal{T}(\textbf{W},p)$. First, sort the edges $(\{i,j\})_{i \neq j}$ of $G(\textbf{W},p)$ by decreasing order of their capacities $(E_{\{i,j\}})_{i \neq j}$. At step $1$, $\mathcal{G}(\textbf{W},p,1) = G(\textbf{W},p)$. If deleting the first edge in the order disconnects a connected component of  $\mathcal{G}(\textbf{W},p,1)$ then keep it. Otherwise, delete it. This gives $\mathcal{G}(\textbf{W},p,2)$. Then, move on to the next edge and do the same. Continue like this by either keeping or deleting each edge consecutively. This procedure ensures that all the graphs $\mathcal{G}(\textbf{W},p,i)$ have the same number of connected components and only their number of cycles decreases. We call this the edges deletion algorithm,  also known as bombing optimization (\citet*{BR07}).
It is easy to see that this procedure yields a forest in the end since we are removing all the cycles from the graph. Moreover, by construction, $\mathcal{T}(\textbf{W},p)$ is the  minimum spanning forest of $G(\textbf{W},p)$ with respect to the capacities $(E_{\{i,j\}})_{i\neq j}$ because we delete the largest capacities first. Hence, $\mathcal{T}(\textbf{W},+\infty)$ is the minimum spanning tree of the complete graph with respect to the capacities $(E_{\{i,j\}})_{i \neq j}$. The diameter of a graph is the largest graph distance (number of edges) between two of its nodes. If the graph is not connected its diameter is the maximum of the diameters of its connected components. The purpose of this paper is the study of the diameter of $\mathcal{T}(\textbf{W},+\infty)$. Before going further in the discussions, we introduce some notations.

\subsection{Notations and definition of the exploration process}
We start by giving another definition of the graphs and defining the exploration processes we will work with. We use exactly the same definition we already used in  \cite{O20}. 
Consider $n \in \mathbb{N}$ vertices labeled $1,2..,n$. For a vector of weights $\textbf{W}=(w_1,w_2,..w_n)$, where $0< w_n \leq w_{n-1} \leq ... \leq w_1$, we create the inhomogeneous random graph associated to $\textbf{W}$ and to $p \leq +\infty$ in the following way: \par 
Each potential edge $\{i,j\}$ is in the graph with probability $1-e^{-w_iw_jp}$ independently from everything else. This gives a random graph that we call the rank-1 inhomogeneous random graph associated to $\textbf{W}$ and $p \leq +\infty$. This construction yields graphs with the marginal distribution of the sequence $(G(\textbf{W},p))_{p \in [0,+\infty]}$ presented previously. We will keep both those two graphs construction methods in mind in this article, and switch between the two depending on our needs.

Before stating the main theorem, we define an exploration process for $G(\textbf{W},p)$ seen as a graph from the sequence $(G(\textbf{W},p))_{p \in [0,+\infty]}$ for a fixed $p$. This process is based on an "horizontal" exploration of the graph, called the breadth-first walk (BFW).  Write:
$$ 
\ell_n = \sum_{i=1}^n w_i,
$$
and recall that the weights depend implicitly on $n$. We say that a tree is spanning a graph if it has the same set of nodes and a subset of its edges. The BFW also naturally yields a spanning forest of $G(\textbf{W},p)$. That is a sequence of spanning trees of the connected component of $G(\textbf{W},p)$. \par 
For each potential edge $\{i,j\}$ recall the definition of  $E_{\{i,j\}}$ from the previous subsection. The BFW operates by steps, define the following sets of vertices. A vertex is always in exactly one of those sets:
\begin{itemize}
    \item $(\mathcal{U}(i))_{1 \leq i \leq n}$ is the sequence of sets of unexplored vertices at each step.
    \item $(\mathcal{D}(i))_{1 \leq i \leq n}$ is the sequence of sets of discovered but not yet explored vertices at each step.
    \item $(\mathcal{F}(i))_{1 \leq i \leq n}$ is the sequence of sets of explored vertices at each step.
\end{itemize}
First, choose a vertex  $i$  with probability: $$\mathbb{P}(v(1)=i) = \frac{w_i}{\ell_n},$$ 
and call it $v(1)$. Let $\mathcal{V}$ be the set of all vertices labels, and $\mathcal{U}(1)=\mathcal{V}\setminus \{v(1)\}$, $\mathcal{D}(1)=\{v(1)\}$. At step $2$, $v(1)$ is explored. The vertices $j$ that are unexplored and such that $E_{\{j,v(1)\}} \leq p $ become discovered but not yet explored. We call them children of $v(1)$.
Let $c(1)$ be the number of children of $v(1)$. Denote their labels by $(v(2),v(3),...,v(c(1)+1)$ in increasing order of their $E_{\{j,v(1)\}}$'s. For $i \geq 1$, denote the set $\{v(1),v(2),...,v(i)\}$ by $\mathcal{V}_i$. Hence, at step $2$ we have:
\begin{itemize}
    \item $\mathcal{U}(2) = \mathcal{V}\setminus \mathcal{V}_{c(1)+1}$.
    \item $\mathcal{D}(2) = \mathcal{V}_{c(1)+1} \setminus \mathcal{V}_{1}$.
    \item $\mathcal{F}(2) =  \mathcal{V}_{1}$.
\end{itemize}
Now, at the step $3$, $v(2)$ becomes explored. The vertices $j$ that are unexplored and such that $E_{\{j,v(2)\}} \leq p $ become discovered but not yet explored. We call them children of $v(2)$, and we denote their labels by $\{v(c(1)+2),v(c(1)+3)...,v(c(1)+c(2)+3)\}$. The BFW continues like this. At step $i+1$ node $v(i)$ becomes explored. If the set of discovered but not yet explored nodes become empty at some step $i$, this means that the exploration of a connected component is finished. In that case, we move on to the next step by choosing a vertex $j$ with probability proportional to its weight $w_j$ among the unexplored vertices (like we did for $v(1)$). We call this a size-biased sampling. 
The spanning forest associated to the the BFW is created by considering each node without a parent as a root, and adding only the edges between parent nodes and their children. We call the trees in that forest the exploration trees. By construction, exploration trees are spanning trees of the connected components of $G(\textbf{W},p)$. We say that a connected component is discovered at step $i$ if its first node discovered by the BFW is $v(i)$. Similarly, we say that a connected component is explored at step $i$ if its last node that becomes explored in the BFW is $v(i)$.\par
Generally, let $c(i)$ be the number of children of the  node labeled $v(i)$. The exploration process associated to the BFW above is defined as follow, for $n-1 \geq i \geq 0$:
\begin{equation*}
\begin{aligned}
    L'_0&=1,\\
    L'(i+1)&= L'_i+c(i+1)-1.
\end{aligned}
\end{equation*}
The reflected exploration process is defined by
\begin{equation*}
\begin{aligned}
    L_0&=1,\\
    L(i+1)&= \max(L_i+c(i+1)-1,1).
\end{aligned}
\end{equation*}

\begin{figure}[!htbp]
\centering
\includegraphics[width=0.8\textwidth]{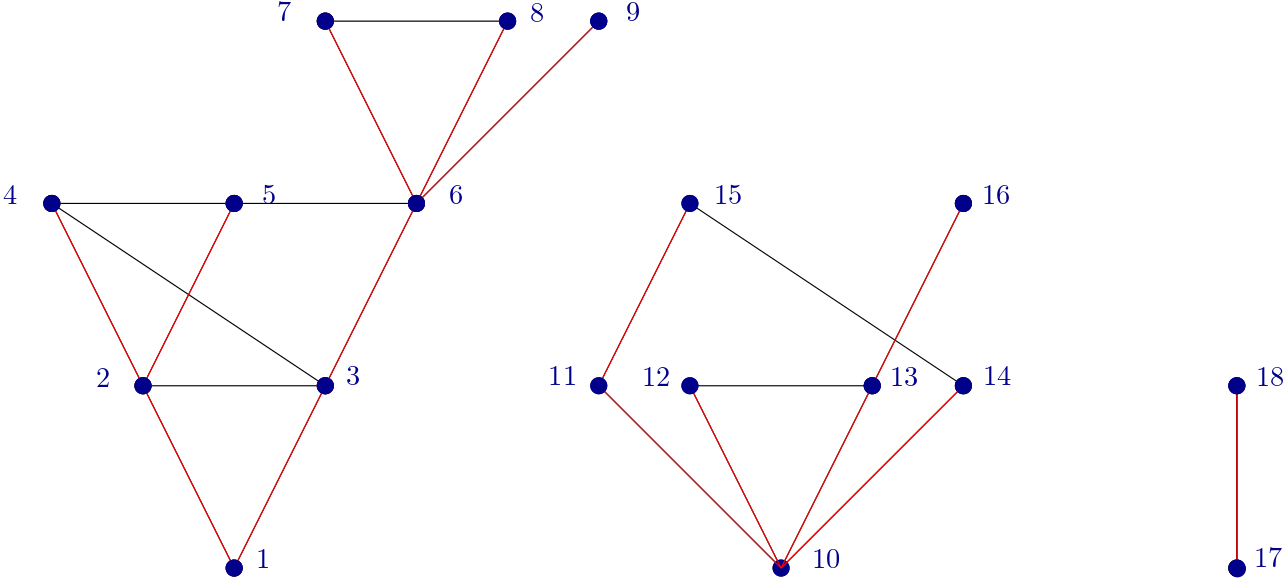}
\caption{Example of a graph with ordered nodes. The integers correspond to the order in the exploration process. The edges in red correspond to the edges of the exploration trees. The labels of the nodes are not represented.}
\label{fig1}
\end{figure}
\begin{figure}[!htbp]
\centering
\includegraphics[width=0.9\textwidth, height=0.3
\textheight]{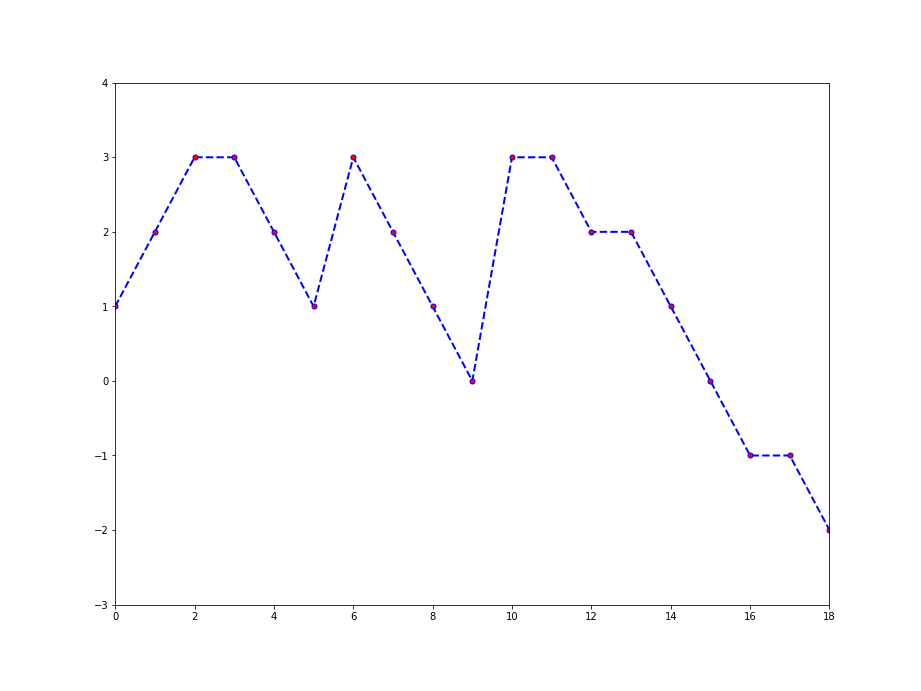}
\caption{The exploration process of the graph in Figure \ref{fig1}.}
\end{figure}

The increment of the process $L'$ at step $i$ is the number of nodes added to the set of discovered nodes in the BFW at step $i$. This number is at least $-1$ if the node being explored has no children. The process $L'$ contains a lot of the information about $G(\textbf{W},p)$. For a more indepth discussion on those process we refer the reader to the related part in \cite{O20}.\par
The order of appearance of the nodes in the exploration process corresponds to a size-biased sampling. Formally, we have:
\begin{Lemma}
With the notations presented above, for:
$$i \in \{0,1,...,n-1\},$$
and $$ j \in  \{0,1,...,n\},$$
we have:
\begin{equation*}
\begin{aligned}
&\mathbb{P}\left(v(1)= j \right) = \frac{w_j}{\ell_n}.  \\
&\mathbb{P}\left(v(i+1)=j \quad | \mathcal{V}_i\right) = \frac{w_j\mathbbm{1}(j \not\in \mathcal{V}_i)}{\ell_n-\sum\limits_{k=1}^{i}w_{v(k)}}.
\end{aligned}
\end{equation*}
\end{Lemma}
The proof of this fact is present in \cite{O20}.
We will assume the following conditions on $\textbf{W}$ in the entire article.

\begin{Conditions} 
\label{cond_nodess}
There exists some positive random  variable $W$  such that: 
\begin{enumerate}[(i)]
\item The distribution of a  uniformly chosen weight $w_{X}$ converges weakly to $W$.
\item $\mathbb{E}[W^3] < \infty$.
\item $\mathbb{E}[W^2] = \mathbb{E}[W]$.
\item $\ell_n = \mathbb{E}[W]n + o(n^{2/3})$.
\item $\sum\limits_{k=1}^{n}w_k^2 = \mathbb{E}[W^2]n + o(n^{2/3})$.
\item $\sum\limits_{k=1}^{n}w_k^3 = \mathbb{E}[W^3]n + o(1)$. 
\item $\max_{i \leq n} w_i = o(n^{1/3})$. 
\end{enumerate}
\end{Conditions}

\begin{Conditions} 
\label{cond_nodess2}
All the Conditions in \ref{cond_nodess} hold, and moreover:
\begin{enumerate}[(i)]
  \setcounter{enumi}{7}
  \item$\sum\limits_{k=1}^{n}w_k^4 = \mathbb{E}[W^4]n + o(1),$ and $\mathbb{E}[W^4] < \infty$.
\end{enumerate}
\end{Conditions}

We refer the reader to  \cite{O20} for a more in-depth discussion of these conditions. Here, we only draw attention to two facts. Firstly, condition $(viii)$ is only present for technical reasons in  \cite{O20}. It is not necessary for the main theorem to hold. And we will get rid of this condition in a subsequent article. Secondly,
an important case to keep in mind is when $(w_1,w_2,...,w_n)$ are realizations of random variables $(W_1,W_2,...,W_n)$ which are i.i.d. with distribution $W$. In that case the conditions can be seen as consequences of convergence theorems and  hold almost surely.\par
We also define the size of a connected component $\mathcal{C}$, with vertices set $V(\mathcal{C})$, as the number of vertices in $\mathcal{C}$. The distance between two vertices of $\mathcal{C}$ is the number of edges in the smallest (in number of edges) path between them. We also define the weight of $\mathcal{C}$ as:
$$
\sum_{i\in V(\mathcal{C}) } w_i.
$$
We call surplus (or excess) of $\mathcal{C}$ the number of edges that have to be removed from it in order to make it a tree. For instance, the surplus of a tree is $0$, and the surplus of a cycle is $1$.
We can now state our two main theorems.

\begin{Theorem}
\label{principal}
Suppose that Conditions \ref{cond_nodess2} are verified. For any $\ell_n^{4/3}-\ell_n^{1/3} \geq f_n \geq 0$ the diameter of $\mathcal{T}(\textbf{W},\frac{1}{\ell_n}+\frac{f_n}{\ell_n^{4/3}})$, denoted by $\textnormal{diam}(\mathcal{T}(\textbf{W},\frac{1}{\ell_n}+\frac{f_n}{\ell_n^{4/3}}))$, verifies\footnote{The notation $\Theta(n^{\frac{1}{3}})$ means that there exists two constants $\epsilon > 0$ and $A > 0$ and an $N > 0$, such that for any $n \geq N$ we have the diameter is larger than $\epsilon n^{1/3}$ and smaller than $A n^{1/3}$.}:
$$\mathbb{E}\left[\textnormal{diam}\left(\mathcal{T}\left(\textbf{W},\frac{1}{\ell_n}+\frac{f_n}{\ell_n^{4/3}}\right)\right)\right] =\Theta(n^{\frac{1}{3}}).$$\\ 
\end{Theorem}
We will further discuss our choice of parameter $f_n$ in this theorem in the next sections.
The second theorem is related to typical distances in those minimum spanning trees.
\begin{Theorem}
\label{stats}
Suppose that Conditions \ref{cond_nodess2} are verified. For  $c > 1$ and a sequence of parameters $ \tilde{p}_n \geq \frac{c}{\ell_n}$, let $(U_{f_n},V_{f_n})$ be two uniformly randomly drawn nodes from the largest tree in $\mathcal{T}(\textbf{W},\tilde{p}_n )$, and let  $d(U_{f_n},V_{f_n})$ be the distance between them in $\mathcal{T}(\textbf{W},\tilde{p}_n )$.  We have : 
$$
\mathbb{E}[d(U_{f_n},V_{f_n})] = \Theta(n^{1/3}).
$$
\end{Theorem}
One important thing that will be clear from our proofs is that those theorems hold if we replace the inhomogeneous minimum spanning trees by other types of trees verifying an inclusion property. 
Let $a > 0$ and $(G_p)_{p \geq a}$ be an increasing process of graphs for inclusion. We say that the sequence of forests $(T_p)_{p \geq a}$ is an increasing process of spanning forests of  $(G_p)_{p \geq a}$ if, for any $p \geq a$, $T_p$ is a spanning forest of $G_p$, and the process $(T_p)_{p \geq a}$ is increasing for inclusion.

By definition $(\mathcal{T}(\textbf{W},p))_{p \geq 0}$ is an increasing process of spanning forests of $(G(\textbf{W},p))_{p \geq 0}$. A direct corollary of our proofs is the following : 
\begin{Corollary}
\label{gen}
Theorems \ref{principal} and \ref{stats} still hold if we replace the trees $(\mathcal{T}(\textbf{W},p))_{p \geq \frac{1}{\ell_n}}$ by any increasing process of spanning forests of $(G(\textbf{W},p))_{p \geq \frac{1}{\ell_n}}$.
\end{Corollary}
This corollary is true because our proofs only use the fact that $(\mathcal{T}(\textbf{W},p))_{p \geq 0}$ in an increasing process of spanning forests of $(G(\textbf{W},p))_{p \geq 0}$. And we never use any other property proper to $(\mathcal{T}(\textbf{W},p))_{p \geq 0}$. \par
This is thus a generalization of Theorems \ref{principal} and \ref{stats}. However, we emphasise the fact that given an increasing process of graphs for inclusion. An increasing process of spanning trees is uniquely determined by the spanning forest of the first graph in the process. The condition of being an increasing process of spanning trees is thus very strong and does not allow for much freedom in the choice of the process. Still, this is a nice addition that will be used in this article to tackle a conjecture from statistical physics. And we hope that it will find use in later work.

\subsection{Further discussion and related work in statistical physics}
\cite{LNB09} already showed a similar result to Theorem \ref{principal} with $\textbf{W}=(1,1,...,1)$. In this case the inhomogeneous random graph is in fact an Erd\H{o}s-Rényi random graphs. Our theorem is thus a generalization of theirs. It is also  another hint at the fact that when the weights of rank-1 inhomogeneous random graphs verify a third moment condition, which corresponds to Conditions \ref{cond_nodess}, then their asymptotic behavior is similar to Erd\H{o}s-Rényi random graphs. Meaning that the inhomogeneity disappears asymptotically. Such behavior was already remarked for other asymptotic properties (see for instance \cite{SRJ10}). \par
Another interesting result is in \cite{LNCG13} where the authors showed the existence of a measured metric space that is a scaling limit of the minimum spanning tree associated to Erd\H{o}s-Rényi random graph with distances rescaled by $n^{-1/3}$. In that article the concentration results that gave the  order of the diameter of the minimum spanning tree corresponding to $\textbf{W}=(1,1,...,1)$ were crucially used. A similar result was obtained for $3$-regular graphs with i.i.d. capacities on their edges in \citet*{addario2018geometry}. We will do the same for general weights verifying Conditions \ref{cond_nodess} in a subsequent article. \par

In an even more general setting, one can ease the Conditions \ref{cond_nodess2}. A large body of work (\citet*{CHE06}, \cite{BR07}, \citet*{RJJ17}, \citet*{RSJ18}, \citet*{br20}) is interested in weights related to power law distributions, the so called scale-free model. In that case the diameter is not expected to be of order $n^{1/3}$. If we suppose that the weights are drawn from a distribution with power law tail with parameter $4 \geq \alpha > 3$, then intuitive arguments suggest that the diameter of the minimum spanning tree should be of order $n^{\frac{\alpha-3}{\alpha-1}}$ (see \cite{br20}). The scaling limit of such trees should also be mutually singular from one another for different values of $\alpha$. Some of the tools presented here could be useful to prove results for weights with power law distributions. However, more work is required to prove those conjectures.  \par

The model presented here is closely related and generally called the Norros-Reitu model (\citet*{no06}) altough it is slightly different from the first model proposed by Norros and Reitu, since their model allows for multigraphs. This has no incidence on our results. In fact, our model is more closely related to the multiplicative coalescent introduced by Aldous in \citet*{AL97}, and further studied in \citet*{AL98} by Aldous and Limic. Moreover,  the results we present for this model also hold for equivalent models, in the sense of \citet*{janson2010asymptotic}. For instance, Chung-Lu (\citet*{CL06}) proposed a random graph with prescribed expected degrees. In that model two nodes $\{i,j\}$ are connected with probability:
$$
p_{i,j} = \frac{w_iw_j}{\ell_n}.
$$
It supposes that $\max_{i,j}(w_iw_j) \leq \ell_n$.
The following classical theorem (proved in \citet*{BSO07}) holds for all those graph models and is crucial in understanding the evolution of the graph process $(G(\textbf{W},p))_{p \geq 0}$. \par
\begin{Theorem}
\label{phase_trans}
Recall that $\tilde{p}_n = \frac{c}{\ell_n}$. Take $G(\textbf{W},\tilde{p}_n)$ and  suppose that Conditions \ref{cond_nodess} are verified, then the following results hold with high probability \footnote{We say that a sequence of events $E_n$ holds with high probability if $\lim_{n \rightarrow \infty}\mathbb{P}(E_n) =1$.}:
\begin{itemize}
\item \textbf{Subcritical regime}  If $c < 1$ then the largest connected component is of size $o(n)$.
\item \textbf{Supercritical regime}  If $c > 1$ then the largest connected component is of size $\Theta(n)$ and for any $i >1$ the $i$-th largest connected component is of size $o(n)$.
\item \textbf{Critical regime}  If $c = 1$ then for any $i \geq 1$ the $i$-th largest connected component is of size $\Theta(n^{2/3})$.
\end{itemize}
\end{Theorem}
A simulation-based conjecture from statistical physics (\cite{CHE06}, \cite{WU06}, \cite{BR07}) is the following: Consider the largest component of $G(\textbf{W},\tilde{p}_n)$ for $c > 1$. Put i.i.d. capacities  derived from some continuous distribution on the edges of that component. Then typical distances in the minimal spanning tree constructed on the largest component with those i.i.d. capacities scale like $n^{1/3}$. Here typical distances mean distance between two uniformly drawn nodes from the tree. First, notice that as long as the distribution used is atomless, it has no incidence on the distribution of the geometry of the tree obtained. Indeed, only the relative order of the edge capacities is relevant for the geometry of the minimum spanning tree. Hence, putting i.i.d. capacities on the edges amounts to taking a uniform random order on them for our purpose. In order to prove this conjecture, it is sufficient to prove a result similar to Theorem \ref{stats} but for the type of minimum spanning trees present in the statistical physics literature. In fact, this minimum spanning tree is not that much different from the inhomogeneous minimum spanning trees that we study in this article. One can see that, with a little more work, the result from \cite{janson2010asymptotic} can be used to show that if Theorem \ref{stats} is true, then a similar theorem also holds for the minimum spanning trees constructed from i.i.d capacities on the giant component of supercritical inhomogeneuous random graphs. However, for clarity, after proving Theorem \ref{stats}, we will prove a modified version of it to show the conjecture from statistical physics using our results. We thus prove the following Theorem:
\begin{Theorem}
\label{stats_2}
Suppose that Conditions \ref{cond_nodess2} are verified. For  $c > 1$, and a sequence of parameters $ \tilde{p}_n \geq \frac{c}{\ell_n}$, put i.i.d. capacities with an atomless distribution on the edges of  the graphs $(\mathcal{G}(\textbf{W},\tilde{p}_n))_{n \geq 1}$. Let $\mathcal{T}'(\textbf{W},\tilde{p}_n )$ be the minimum spanning forest of  $\mathcal{G}(\textbf{W},\tilde{p}_n)$ corresponding to those edge capacities. Let
$(U_{f_n},V_{f_n})$ be two uniformly randomly drawn nodes from the largest tree in $\mathcal{T}'(\textbf{W},\tilde{p}_n )$, and let  $d(U_{f_n},V_{f_n})$ be the distance between them in $\mathcal{T}'(\textbf{W},\tilde{p}_n )$.  We have : 
$$
\mathbb{E}[d(U_{f_n},V_{f_n})] = \Theta(n^{1/3}).
$$
\end{Theorem}
Another closely related conjecture concerns the power law distribution case. It is generally also conjectured that typical distances in the minimum spanning tree with i.i.d. capacities on the edges of largest components of graphs with nodes weights following a distribution with a power law tail of parameter $4 > \alpha > 3$, either in inhomogeneous random graphs as studied here, or in the configuration model (\cite{BR07}), scale like $n^{\frac{\alpha-3}{\alpha-1}}$. Proving the existence of a non-trivial scaling limits for $\mathcal{T}(\textbf{W},+\infty)$ when we change Conditions \ref{cond_nodess} to conditions pertaining to power law distribution would also prove such a conjecture. A first step in that direction was done very recently in \citet*{SB20}. The authors proved that the diameter of such trees is in fact of order $n^{\frac{\alpha-3}{\alpha-1}}$ in the inhomogeneous case and they even showed the existence of a non-trivial scaling limit for those trees. However, their work supposes restrictive conditions on the distribution other than just having power law tail. Thus, the general case remains open. \par

\begin{figure}[!tbp]
  \centering
  \begin{minipage}[b]{0.8\textwidth}
    \includegraphics[width=\textwidth]{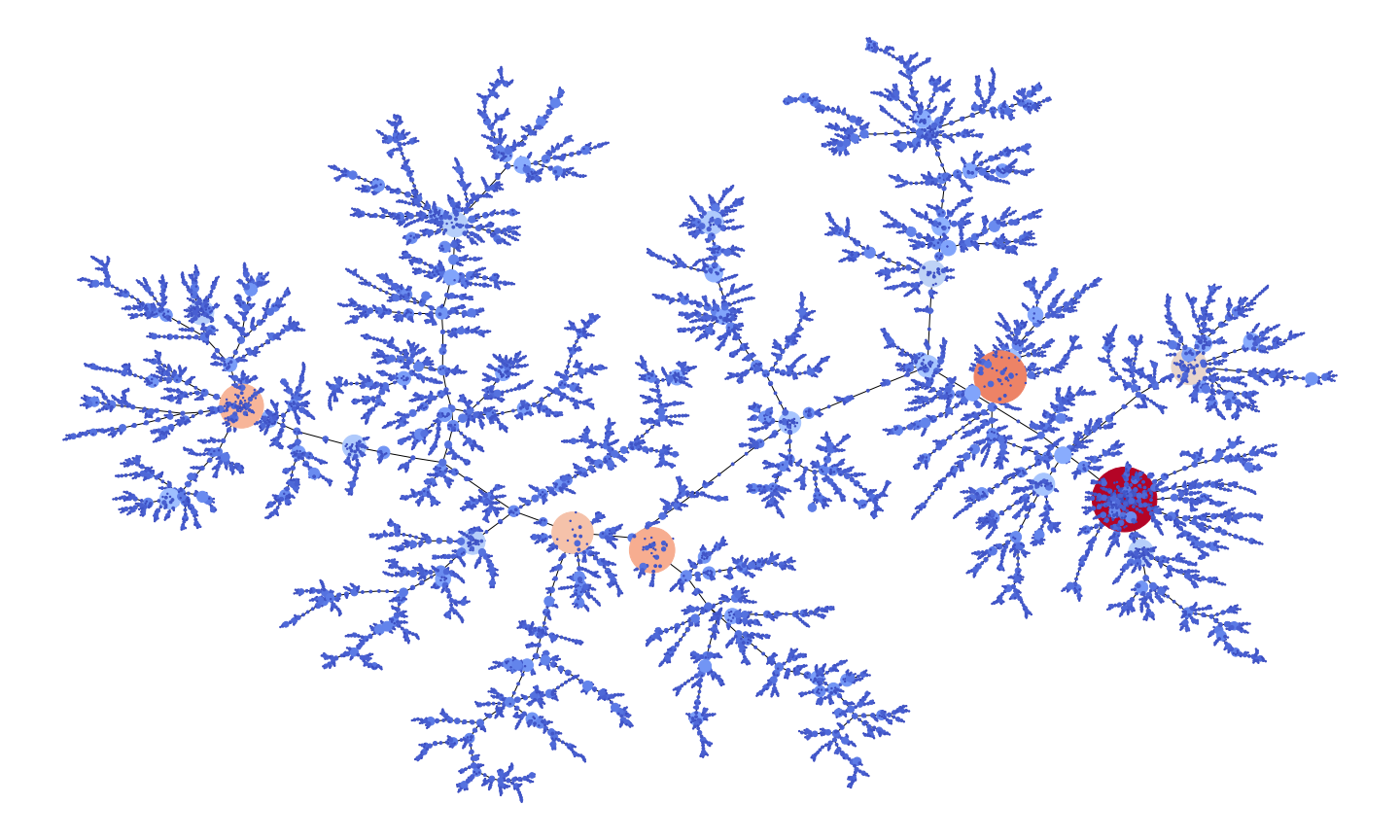}
    \caption{A minimum spanning tree built on the giant component of a scale-free supercritical inhomogeneous random graph with parameter $\alpha = 3.5$. The edge capacities are i.i.d.. Nodes are coloured from largest in red to smallest in dark-blue. Node sizes are proportionnal to the square of their degrees.}
  \end{minipage}
  \hfill
  \begin{minipage}[b]{0.8\textwidth}
    \includegraphics[width=\textwidth]{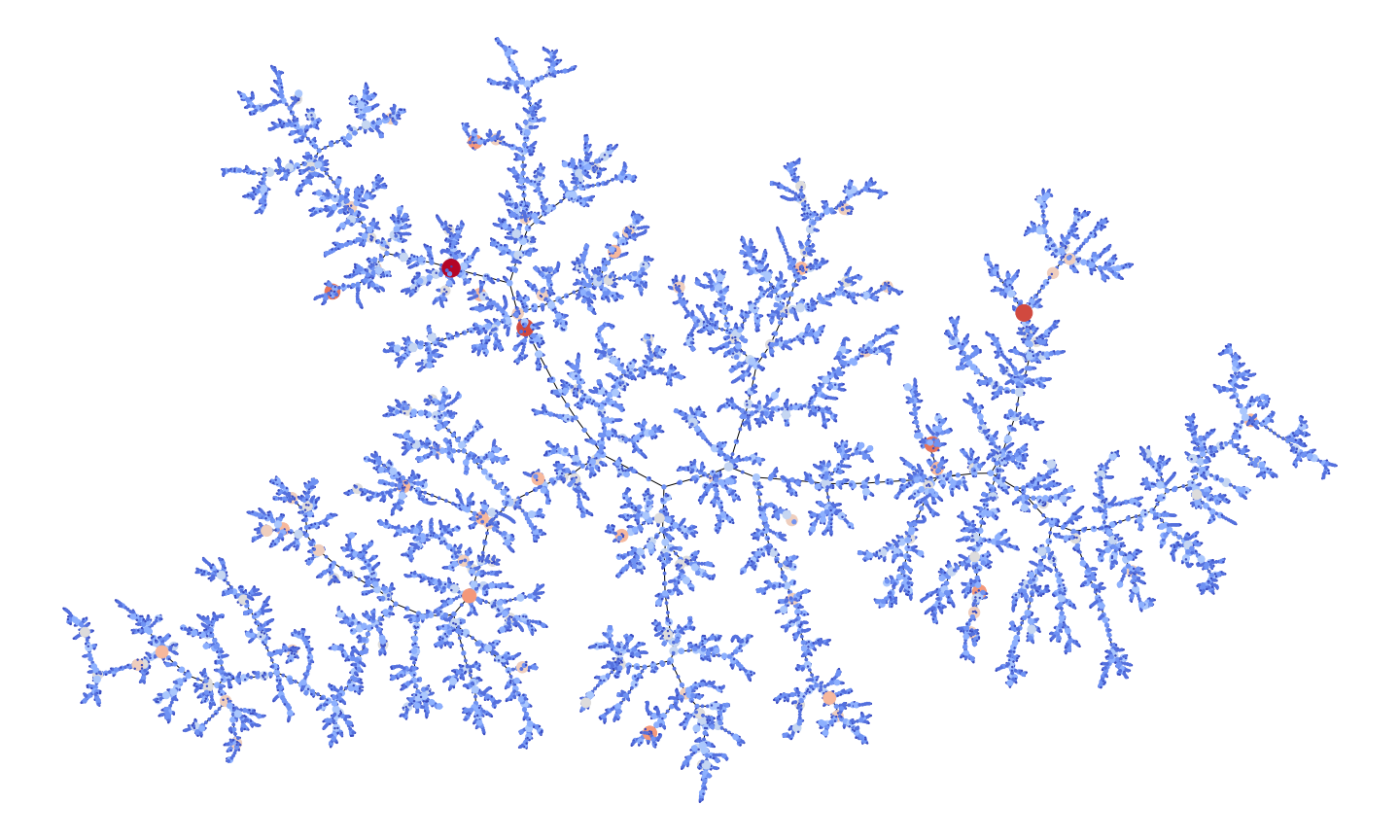}
    \caption{A minimum spanning tree built on the giant component of a  supercritical inhomogeneous random graph with node weights having finite third moments. The edge capacities are i.i.d.. Nodes are coloured from largest in red to smallest in dark-blue. Node sizes are proportionnal to the square of their degrees.}
  \end{minipage}
\end{figure}
Lastly in this introduction, we remind the reader of Bernstein's inequality (\citet*{B24}, \citet*{BML13}). This inequality will be used through this article in order to prove some concentration bounds. 
\begin{Lemma}
Let $X_1,X_2,...,X_n$ be i.i.d. random variables with finite second moment. Suppose that $|X_i| \leq M$, for all $i$. Then for any positive number $x$:
$$
\mathbb{P}\left(\middle|\sum_{i\geq 1}^n(X_i-\mathbb{E}[X_i])\middle| \geq x\right) \leq 2\exp \left ( \frac{-\tfrac{1}{2} x^2}{\sum_{i = 1}^n \mathbb{E} \left[X_i^2 \right ]+\tfrac{1}{3} Mt} \right ).
$$
\end{Lemma}

\section{First ingredients of the proof}

For $p \geq 0$. Any component of $G(\textbf{W},p)$ that is a tree is also a component of $\mathcal{T}(\textbf{W},p)$. Since $\mathcal{T}(\textbf{W},p) \subseteq \mathcal{T}(\textbf{W},+\infty)$, the diameter of any tree component of $G(\textbf{W},p)$ is a lower bound on $\textnormal{diam}(\mathcal{T}(\textbf{W},+\infty))$.
The following theorem is a simple corollary of the much more general convergence result in Theorem $2.4$ of \cite{br20}.
\begin{Theorem}
\label{lower_bound}
There exists $\epsilon > 0$ and $\epsilon' >  0$ such that for all $n$ large enough the largest connected component of $G(\textbf{W},\frac{1}{\ell_n})$ is a tree of diameter at least  $\epsilon' n^{1/3}$ with probability at least $\epsilon$.
\end{Theorem}
Theorem \ref{lower_bound} proves that there exist a constant $\epsilon>0$  such that $\mathbb{E}[\textnormal{diam}(\mathcal{T}(\textbf{W},+\infty))] \geq \epsilon n^{1/3}$. The rest of this paper is dedicated to proving the upper bound of Theorem \ref{principal}.
\subsection{The phase transition}
The edge deletion algorithm provides a coupling between $\mathcal{T}(\textbf{W},+\infty)$ and the graph 
$G(\textbf{W},+\infty)$. Hence, by studying the evolution of the diameter of the graphs in  
$(G(\textbf{W},p))_{p \in [0,+\infty]}$ we will be able to prove Theorem \ref{principal}. As a first step, we want to know what are the important values of $p$ in this process. That question was partly answered in Theorem \ref{phase_trans}. We see that there is a phase transition at $p=\frac{1}{\ell_n}$. It turns out that $p = p_f = \frac{1}{\ell_n}+\frac{f}{\ell_n^{4/3}}$, with $f \in \mathbb{R}$, constitutes a phase transition "window" for the size of the connected components sizes. In that window, the size and weight of the largest component increase smoothly. This is the so called critical regime. It is thus natural to investigate this critical regime in order to prove Theorem \ref{principal}. \par
We will show that the diameter of  $\mathcal{T}(\textbf{W},p_f)$ is a $O(n^{1/3})$, when $p = p_f = \frac{1}{n}+\frac{f}{n^{4/3}}$ with $f >0$ a large constant. Then it only grows by $O(n^{1/3})$ when $p$ becomes strictly larger than $\frac{1}{n}$.
The increase of the diameter in the supercritical phase is easy to study, the main problem is to show that when we cross the critical regime, $p_f = \frac{1}{\ell_n}+\frac{f_n}{\ell_n^{4/3}}$ with $f_n$ going to infinity, the diameter remains of order $n^{1/3}$. The value $f'_n = \frac{\ell_n^{1/3}}{\log(n)}$ which gives  $p_{f'_n} = \frac{1}{\ell_n}+\frac{1}{\log(n)\ell_n}$  constitutes the pivotal moment where we will switch arguments: $p_{f'_n}$ is chosen like this because it is neither a critical nor a real supercritical value. We say that it is in the barely supercritical regime. 
In the rest of this  section we will take Theorems \ref{well-behaved} and \ref{critical-diameter} for granted and prove Theorem \ref{principal}. The proofs of these theorems will be the focus of the rest of the paper. \par
\textbf{Notation}: In the remainder of the article we drop the $n$ from $f_n$ for clarity, $f$ will always be the critical parameter, moreover we will always assume $f = o(n^{1/3})$ and $f \geq F$, where $F >0$ is a constant independent of $n$ which is large enough for all our theorems to be true. Similarly $A, A', A'',... \in \mathbb{R}^+$ will always be large constants independent of $n$. And $\epsilon, \epsilon' , ... \in \mathbb{R}^+$ will always be small constants independent of $n$.
\begin{Theorem}
\label{well-behaved}
Assume Conditions \ref{cond_nodess2}, there exists a  positive constant $A >0$  such that: With probability at least $(1-\frac{1}{n})$, the largest component, $H(\textbf{W},p_{f'_n})$, of $G(\textbf{W},p_{f'_n})$ has total weight at least $\frac{n}{A\log(n)}$ and every other connected component has total weight at most $A\log(n)^{1/2}n^{1/2}$ and longest path  of length less than $A\log(n)^{1/4}n^{1/4}$.
\end{Theorem}
\begin{Theorem}
\label{critical-diameter}
Assume Conditions \ref{cond_nodess2},  then:
$$ \mathbb{E}[\text{diam}(H(\textbf{W},p_{f'_n}))] = O(n^{1/3})$$
\end{Theorem}

\subsection{The supercritical phase}
Here we take Theorems \ref{well-behaved} and \ref{critical-diameter} for granted. Let $\mathcal{C}$ be a connected component, with vertices set $V(\mathcal{C})$  and edge set $D(\mathcal{C})$ , and suppose that $\mathcal{C}$ is not the largest connected component of $G(\textbf{W},p_{f'_n})$. Let $\{I,J\}$ be the random edge with exactly one endpoint (say $I$) in $\mathcal{C}$  and with minimal capacities $E_{\{I,J\}}$ among edges with exactly one endpoint in $\mathcal{C}$. By construction, $\{I,J\}$ will necessarily be in $\mathcal{T}(\textbf{W}; +\infty)$. We want to calculate $\mathbb{P}(J \in V(H(\textbf{W},p_{f'_n})))$, the probability that the other endpoint of the edge $\{I,J\}$ lies in the largest component $H(\textbf{W},p_{f'_n})$ with vertices set $V(H(\textbf{W},p_{f'_n}))$ and edge set $D(H(\textbf{W},p_{f'_n}))$.
Recall that for any couple $\{i,j\}$, $E_{\{i,j\}}$ is an exponential random variable of parameter $w_iw_j$. So if we fix $\{i,j\}$, by  classical results on exponential random variables: 
$$ \mathbb{P}\left(\forall l \neq i, l \neq j \quad E_{\{i,j\}}  < E_{\{i,l\}} \right) = \frac{w_j}{\sum_{k=1}^{n}w_k}.$$
Similarly, write:
$$
\textbf{U} = \left\{i \in V(\mathcal{C}), \exists j \in V(H(\textbf{W},p_{f'_n})), \forall l \in (\mathcal{V}\setminus (V(\mathcal{C}) \cup \{j\}) \; E_{\{i,j\}}  < E_{\{i,l\}}\right\}.
$$
Then:
\begin{equation*}
\begin{aligned}
&\mathbb{P}\left( \textbf{U}   |   (V(H(\textbf{W},p_{f'_n})), V(\mathcal{C}))\right)  = \frac{\sum_{k\in V({H(\textbf{W},p_{f'_n}))}}{w_k}}{\sum_{k'=1}^{n}w_{k'}-\sum_{k''\in V(\mathcal{C})}w_{k''}}.
\end{aligned}
\end{equation*} 
Which implies that: 
\begin{equation}
\label{eq1}
\mathbb{P}\left(J \in V(H(\textbf{W},p_{f'_n})) |  (V(H(\textbf{W},p_{f'_n})), V(\mathcal{C}))\right)  = \frac{\sum_{k\in {H(\textbf{W},p_{f'_n})_V}}{w_k}}{\sum_{k'=1}^{n}w_{k'}-\sum_{k''\in V(\mathcal{C})}w_{k''}}.
\end{equation}
By Theorem \ref{well-behaved} and Equation \eqref{eq1}, we get for $n$ large enough:
\begin{equation*}
\begin{aligned}
1- \mathbb{P}\left(J \in V(H(\textbf{W},p_{f'_n}))\right) 
&\leq 1-\frac{n}{A\log(n)(n-A\log(n)^{1/2}n^{1/2})}-\frac{1}{n} \\
&\leq 1-\frac{2}{A\log(n)},
\end{aligned}
\end{equation*}
If $J$ is not in $V(H(\textbf{W},p_{f'_n}))$, then $J$ lies in another connected component $\mathcal{C}'$, with vertix set $V(\mathcal{C}')$ and edge set $D(\mathcal{C}')$. So the longest path of the newly created component will be at most $2A\log(n)^{1/4}n^{1/4}$ with probability at least $1-\frac{2}{A\log(n)}$.
Let $\{I',J'\}$ be the edge with the smallest capacities with exactly one endpoint $I'$ in  the newly created connected component $\mathcal{C}''$ which is the concatenation of $\mathcal{C}$ and $\mathcal{C}'$ through the edge $\{I,J\}$.  By the same argument as before, conditionally on the event $F$ that $J$ is not in $V(H(\textbf{W},p_{f'_n}))$, we get for $n$ large enough:
\begin{equation*}
\begin{aligned}
1-\mathbb{P}\left(J' \not \in V(H(\textbf{W},p_{f'_n}))) | F\right) 
&\leq 1-\frac{n}{\log(n)(n-2A\log(n)^{1/2}n^{1/2})}- \frac{1}{n} \\
&\leq 1-\frac{2}{A\log(n)}.
\end{aligned}
\end{equation*}
By an immediate induction, for any $r_n$ small enough, the  probability that $\mathcal{C}$ is in a component of longest path larger than $r_nA\log(n)^{1/4}n^{1/4}$ when it connects to the largest connected component is at most $(1-\frac{2}{A\log(n)})^{r_n}$.
By taking $r_n = A(\log(n))^2$, it follows, using the inequality $$\left(1-\frac{2}{A\log(n)}\right) \leq \exp\left(\frac{-2}{A\log(n)}\right),$$ that the probability that the longest path of $\mathcal{C}$ reaches $A^2\log(n)^{9/4}n^{1/4}$ before connecting to the largest component is less than $\frac{1}{n^2}$. 
Since $\mathcal{C}$ was arbitrary, and since there are at most $n$ such components. We get by the union bound that the probability that there is a component that has a path longer than $A^2\log(n)^{9/4}n^{1/4}= o(n^{1/3})$ when it connects to the largest connected component is at most $\frac{1}{n}$. 
\bigbreak

In order to link the diameters of nested graphs we use the following lemma (Also used in \cite{LNB09}).
\begin{Lemma}
\label{croiss}
Let $H'$ and $H$ be two connected graphs such that $H' \subset H$. Then 
$$diam(H) \leq diam(H') + 2lp(H(V - V(H'))) + 2.$$
Here $lp$ stands for longest path and $H(V - V(H'))$ is the graph $H$ from which we deleted the nodes of $H'$ and any edge that had at least one endpoint in $H'$.
\end{Lemma}
\begin{proof}
In order to show the lemma it is sufficient to show that between any two vertices $H$ there exist a path of length at most than $diam(H') + 2lp(H(V - V(H'))) + 2$.
Let $x,y$ be two vertices of $H$, one can create a path from $x$ to $y$ in the following way. 
\begin{itemize}
    \item There exists a path $\pi_1$  between $x$ and a vertex of $H'$ that has length at most $ lp(H(V - V(H'))) + 1$, let $x'$ be the arrival vertex of this path. 
    \item There exists a path $\pi_2$  between $y$ and a vertex of $H'$ that has length at most $ lp(H(V - V(H'))) + 1$, let $y'$ be the arrival vertex of this path. 
    \item There exists a path $\pi_3$ between $x'$ and $y'$ inside $H'$ that has length at most $diam(H')$. 
\end{itemize}

Concatenating the three paths $\pi_1$, $\pi_2$ and $\pi_3$ provides the desired path.

\end{proof}
Lemma \ref{croiss} and the discussion above imply that: 
\begin{Theorem}
\label{final-theorem}
Under conditions \ref{cond_nodess2}:
$$ \mathbb{E}\left[\textnormal{diam}(\mathcal{T}(\textbf{W}; +\infty))-\textnormal{diam}(\mathcal{T}(\textbf{W}; f'_n))\right] =  O(\log(n)^{9/4}n^{1/4}).$$
\end{Theorem}
\noindent Theorem $1$ follows from this theorem and Theorem \ref{critical-diameter}.
Now what is left is proving Theorems \ref{well-behaved} and \ref{critical-diameter}. This is what we will do in the rest of the paper. 

\subsection{Known results}
Here we remind the reader of the main theorems in  \cite{O20} that will be used in this article. Let:
$$
C = \frac{\mathbb{E}[W^3]}{\mathbb{E}[W]} \geq 1.
$$
\begin{Theorem}
\label{principal_a}
There exists a constant $A >0$, such that for any $1 > \epsilon' > 0$, $1 \geq \epsilon >0$,  and $f = o(n^{1/3})$ large enough with probability at least:
$$
1-A\exp\left(\frac{-f^{\epsilon/2}}{A}\right),
$$
all the following events occur in $G(\textbf{W},p_{f})$:

\begin{itemize}
    \item The size of the largest component is in the interval $$\left[\frac{2(1-\epsilon'/2)f\ell_n^{2/3}}{C},\frac{2(1+\epsilon'/2)f\ell_n^{2/3}}{C}\right].$$
    \item Every other connected component has size less than $\frac{A\ell_n^{2/3}}{f^{1-\epsilon}}$.
    \item The weight of the largest component is in the interval $$\left[\frac{2(1-\epsilon')f\ell_n^{2/3}}{C},\frac{2(1+\epsilon')f\ell_n^{2/3}}{C}\right].$$
    \item Every other connected component has weight less than $\frac{(1+\epsilon')A\ell_n^{1/3}}{f^{1-\epsilon}}$.
    \item The surplus of the largest connected component is less than $Af^3$.
    \item The surplus of any other connected component is less than $Af^{\epsilon}$.
\end{itemize}
\end{Theorem}

This theorem contains many important statements for our proof. It gives bounds on the size, weight and surplus of the largest connected component of rank-1 inhomogeneous random graphs in the critical window, and also of its small components. For instance, by taking $f = f'_n$ in the above theorem we get the statement about the weights in Theorem \ref{well-behaved}. The analysis in \cite{O20} also provides bounds on the time the giant component is discovered during the exploration process. These results show a large part of Theorems \ref{well-behaved} and \ref{critical-diameter}. However they are not enough to bound the diameter of the graph. Information about the length of paths in the graph is still lacking.
Although Theorem \ref{principal_a} contains a bound on the surplus of each connected component of $G(\textbf{W},p_{f'_n})$, it is not sufficient to prove upper bounds on the height of any spanning trees of the connected components in order to finish the proof. This is due to the fact that the surplus of the connected components of $G(\textbf{W},p_{f'_n})$ is too large. In the remainder of the paper we will use a "snapshot" trick to fix this problem.
\section{Dealing with the small components}

In order to bound the diameter of a graph using its excess and its height we will use the following lemma.
\begin{Lemma}
\label{excess_height}
Let  $G$ be a connected graph of excess $q$ and such that there exist a spanning tree $T$ of this graph of height $h$, then the longest path of the graph is of length at most  
$$
2h(q+1)+q.
$$
\end{Lemma}
\begin{proof}
Let $\{D_1,D_2,D_3,....\}$ be the set of edge-disjoint paths of maximum lengths in $T$ taken in decreasing order of their lengths, if two paths of the same length exist at some point we choose one of them arbitrarily. We can create $G$ by adding the excess edges consecutively to $T$. \\

Recursively, the first edge added will create a new longest path of length at most $|D_1|+|D_2|+1 \leq 4h+1$. The second edge added will at most create a new longest path containing $D_1$,$D_2$ and $D_3$ and of length $|D_1|+|D_2|+|D_3|+2 \leq 6h+2$. Hence, a direct induction shows that after adding $q$ excess edges the longest path we can have is of length $2h(q+1)+q$.
\end{proof}
If we want to use this lemma, then bounding directly the height of the exploration tree of the largest component of  $G(\textbf{W},p_{f'_n})$ will not be enough to prove Theorem \ref{critical-diameter}. Indeed, Theorem \ref{principal_a} states that the excess of such a component is upper bounded by $f_{n}^{'3}= \frac{\ell_n}{\log(n)^3}$. This is already much larger than the $n^{1/3}$ of Theorem \ref{critical-diameter}. In order to circumvent this problem we will use the following steps. \par
\begin{enumerate}
\item Define a sequence of "snapshots", $(p_{f(i)})_{i \geq 0}$ such that: $$f(0) = F$$ is a  large constant, and 
$$f(i+1) = \frac{3}{2}f(i),$$
for any $i \geq 0$. 
\item For every $i \geq 0$, let $\tilde{\mathcal{V}}_i$ be the set of nodes of the largest connected component of $G(\textbf{W},p_{f(i+1)})$, and let $G(\mathcal{V} \setminus \tilde{\mathcal{V}}_i,p_{f(i+1)})$ be the sub-graph of $G(\textbf{W},p_{f(i+1)})$ from which we have taken out the nodes of the largest component of $G(\textbf{W},p_{f(i)})$ alongside any edge that has one of those nodes as endpoint. Show that the diameter of this graph is bounded by $\frac{A\ell_n^{1/3}}{f(i+1)^{1/4}}$ w.h.p. This will also prove Theorem \ref{well-behaved}.
\item Show that the diameter of the largest component of $G(\textbf{W},p_{f(i)})$ is smaller than $Af^5\ell_n^{1/3}$ w.h.p.
\item Use the precedent steps to show that when moving from $G(\textbf{W},p_{f(i)})$ to $G(\textbf{W},p_{f(i+1)})$, the diameter of the graph does not increase significantly. This will prove Theorem \ref{critical-diameter}.
\end{enumerate}
We provide the details for step $2$ in the present section, step  $3$ in section \ref{sec4}, and step $4$ in the last section.
\subsection{Construction of the growth of small components}
For $i \geq 1$, we want to bound the size, excess and diameter of the graph comprised of $G(\textbf{W},p_{f(i+1)})$  from which we have taken out the vertices of the largest component in $G(\textbf{W},p_{f(i)})$. Recall that we denote the set of those vertices by $\tilde{\mathcal{V}}_i$, and that graph by $G(\mathcal{V} \setminus \tilde{\mathcal{V}}_i,p_{f(i+1)})$. Bounding this diameter will allow us to bound the growth of the giant component when moving from $G(\textbf{W},p_{f(i)})$ to $G(\textbf{W},p_{f(i+1)})$. Instead of studying $G(\mathcal{V} \setminus \tilde{\mathcal{V}}_i,p_{f(i+1)})$ from scratch, we will show that w.h.p, that graph resembles the graph $G(\mathcal{V} \setminus \tilde{\mathcal{V}}_{i+1},p_{f(i+1)})$ which is a critical inhomogeneous random graph from which we have taken the largest component, and which has already been studied extensively in   \cite{O20}. 
Let $K(i,\epsilon')$ be the following event: 
\begin{enumerate}
    \item The size of the largest component of $G(\textbf{W},p_{f(i)})$ is in the interval
$$ \left[ \frac{2(1-\epsilon'/2)f(i)\ell_n^{2/3}}{C}-\frac{\ell_n^{2/3}}{f(i)C}, \frac{2(1+\epsilon'/2)f(i)\ell_n^{2/3}}{C}\right].$$
\item The total weight of the largest connected component of $G(\textbf{W},p_{f(i)})$ is in the interval
$$
\left[\frac{2(1-\epsilon')f(i)\ell_n^{2/3}}{C},\frac{2(1+\epsilon')f(i)\ell_n^{2/3}}{C}\right].
$$
\end{enumerate}
By Theorems $1$ from  \cite{O20}, there exists a large constant $A > 0$  such that $K(i,\epsilon')$  happens with probability at least:
\begin{equation}
\label{eq8}
1-A\exp\left(\frac{-f(i)}{A}\right).
\end{equation}
Let $(\tilde{v}(1),\tilde{v}(2),..)$ be the ordered labels in the exploration process of the graph $G(\mathcal{V} \setminus \tilde{\mathcal{V}}_i,p_{f(i+1)})$. Conditionally on the nodes of the largest component $\tilde{\mathcal{V}}_i$, $G(\mathcal{V} \setminus \tilde{\mathcal{V}}_i,p_{f(i+1)})$ is an inhomogeneous random graph with vertices label set $\mathcal{V} \setminus \tilde{\mathcal{V}}_i$ and probability transition $p_{f(i+1)}$.   Moereover,  since the event $K(i,\epsilon')$ is measurable with regard to $\tilde{\mathcal{V}}_i$. The same proofs of Lemmas $7$, $8$, $9$, $10$ and $33$ from  \cite{O20} also yields similar results if we replace the $v(j)$'s with $\tilde{v}(j)$'s. We prove one of those results as an example here, and direct the reader to  \cite{O20} for the rest of the proofs. Let $t = \frac{2(1+\epsilon'/2)f(i)\ell_n^{2/3}}{C}$.
\begin{Lemma}
\label{sum_capacities}
For any $n-t \geq l > 0$,  we have:
$$\mathbb{P}\left(\mathbbm{1}(K(i,\epsilon'))\sum_{k=1}^{l}w_{\tilde{v}(k)}\geq 2l^{1/2}\ell_n^{1/2}\right) \leq \exp\left({\frac{-l^{1/2}n^{1/6}}{4}}\right).$$
\end{Lemma}
\begin{proof}
Let $\tilde{J}(1),..,\tilde{J}(l)$ be i.i.d. copies of $\tilde{v}(1)$ conditionnally on $K(i,\epsilon')$. Recall that, by Conditions \ref{cond_nodess},  $\max_{i\leq n}(w_i) = o(n^{1/3})$.
Write $C_l = l^{1/2}\ell_n^{1/2}$,  since the event $K(i,\epsilon')$ is measurable with regard to $\tilde{\mathcal{V}}_i$, we can apply Theorem $1$ in \citet*{AYS16} and Markov's inequality to obtain:
\begin{equation}
\begin{aligned}
\label{rer}
\mathbb{P}\left(\sum_{k=1}^{l}w_{\tilde{v}(k)}\geq 2C_l \middle| K(i,\epsilon')\right) &\leq \frac{ \mathbb{E}[\exp(\sum_{k=1}^{l}w_{\tilde{v}(k)}) | K(i,\epsilon')]}{\exp(2C_l)} \\
&\leq   \frac{ \mathbb{E}[\exp(\sum_{k=1}^{l}w_{\tilde{J}(i)})]}{\exp(2C_l)}.
\end{aligned}
\end{equation}
Moreover: 
\begin{equation*}
\begin{aligned}
\mathbb{E}\left[w_{\tilde{J}(1)}^2\right] &= \mathbb{E}\left[\left.\sum_{k \not\in \tilde{\mathcal{V}}_i}\frac{w_k^3}{\ell_n-\sum_{k' \in \tilde{\mathcal{V}}_i}w_k'} \right| K(i,\epsilon') \right] \\
&= \mathbb{E}\left[\left.\sum_{k \not\in \tilde{\mathcal{V}}_i}\frac{w_k^3}{\ell_n} \right| K(i,\epsilon') \right](1+o(1)) \\
&= 1-\mathbb{E}\left[\left.\sum_{k \not\in \tilde{\mathcal{V}}_i}\frac{w_k^3}{\ell_n} \right| K(i,\epsilon') \right](1+o(1))+o(1).
\end{aligned}
\end{equation*}
We have by definition of $K(i,\epsilon')$:
$$
\mathbb{E}\left[\left.\sum_{k \not\in \tilde{\mathcal{V}}_i}\frac{w_k^3}{\ell_n} \right| K(i,\epsilon') \right] \leq \sum_{k =1}^{t} \frac{w_k^3}{\ell_n}.
$$
It has already been shown in  \cite{O20} (proof of Lemma $8$) that 
$
\sum_{k =1}^{t} \frac{w_k^3}{\ell_n} = o(1).
$
Thus:
\begin{equation}
\label{rer1}
\mathbb{E}\left[w_{\tilde{J}(1)}^2\right] = 1 + o(1).
\end{equation}
And similarly:
\begin{equation}
\label{rer2}
\mathbb{E}\left[w_{\tilde{J}(1)}\right] = 1 + o(1).
\end{equation}
By Equations \eqref{rer1}, \eqref{rer2} and the  Bernstein's inequality (\cite{B24}) for the $w_{\tilde{J}(i)}$ that stems from the Chernoff bound in Equation \eqref{rer}, we obtain: 
\begin{equation}
\begin{aligned}
\mathbb{P}\left(\mathbbm{1}(K(i,\epsilon'))\sum_{k=1}^{l}w_{\tilde{v}(k)}\geq 2C_l\right) &= \mathbb{P}\left(\sum_{k=1}^{l}w_{\tilde{v}(k)}\geq 2C_l \middle| K(i,\epsilon')\right)\mathbb{P}(K(i,\epsilon'))\\
&\leq \exp\left(\frac{-(-l\mathbb{E}[w_{\tilde{J}(1)}]+2C_l)^2}{2(n^{1/3}C_l+l\mathbb{E}[w_{\tilde{J}(1)}^2])}\right) \\
&\leq \exp\left(\frac{-C_l^2}{4(n^{1/3}C_l)}\right) \\
&\leq  \exp\left(\frac{-l^{1/2}n^{1/6}}{4}\right).
\end{aligned}
\end{equation}
\end{proof}

For $n-t \geq  l \geq 0$ let $E_l = \{\sum_{k=1}^{l}w_{\tilde{v}(k)} < 2l^{1/2}\ell_n^{1/2}\}$. Similarly, the following four lemmas are also  slight modifications of Lemmas $8$, $9$, $10$ and $33$ from  \cite{O20}.
\begin{Lemma}
\label{mean_poww_capacities_2}
Recall that $C = \frac{\mathbb{E}[W^3]}{\mathbb{E}[W]}$.
For any $l=o(n)$, and $l \geq u > 1$:
$$\mathbb{E}(w_{\tilde{v}(u)}^2\mathbbm{1}(K(i,\epsilon')) | E_{l}) = C+o(1),$$
and:
$$\mathbb{E}(w_{\tilde{v}(u)}^2\mathbbm{1}(K(i,\epsilon'))) = C+o(1).$$
\end{Lemma}
\begin{Lemma}
\label{mean_ccapacities}
Let $l=o(n)$. For any integer $l \geq u > 0$ we have:
$$\mathbb{E}(w_{\tilde{v}(u)}\mathbbm{1}(K(i,\epsilon')) | E_{l}) = 1+o(1).$$
\end{Lemma}
\begin{Lemma}
\label{mean_corrr}
Let $l=o(n)$. For any integer $l \geq u > 0$ we have:
$$\mathbb{E}(w_{\tilde{v}(l)}w_{\tilde{v}(u)}) = 1+o(1).$$
\end{Lemma}
\begin{Lemma}
\label{decreasingg}
Let $u' \geq u \geq 0$, then for any $x \geq 0$:
$$
\mathbb{P}(w_{\tilde{v}(u')}\mathbbm{1}(K(i,\epsilon')) \geq x) \leq \mathbb{P}(w_{\tilde{v}(u)}\mathbbm{1}(K(i,\epsilon')) \geq x)
$$
\end{Lemma}

\begin{Lemma}
\label{esperances}
For any $l = o(n)$ we have :
\begin{equation*}
\begin{aligned}
\mathbb{E}[w_{\tilde{v}(l)}\mathbbm{1}(K(i,\epsilon'))] \leq 1+\frac{l+2(1-\epsilon')C^{-1}f(i)\ell_n^{2/3}}{\ell_n}\left(1-C\right) + o\left(\frac{l+f(i)n^{2/3}}{n}\right).
\end{aligned}
\end{equation*}
\end{Lemma}

\begin{proof}
First,. Denote $\{\tilde{v}(1),\tilde{v}(2),...,\tilde{v}(l)\}$ by 
$\textbf{V}_{l-1}$. Then by definition:
\begin{equation*}
\begin{aligned}
\mathbb{E}[w_{\tilde{v}(l)}\mathbbm{1}(K(i,\epsilon'))] &= \mathbb{E}\left[\mathbb{E}[w_{\tilde{v}(l)}\mathbbm{1}(K(i,\epsilon')) | (\textbf{V}_{l-1}, \tilde{\mathcal{V}}_i)]\right] \\
&= \mathbb{E}\left[\mathbbm{1}(K(i,\epsilon'))\sum_{j \not\in \textbf{V}_{l-1} \cup \tilde{\mathcal{V}}_i} \frac{w_j^2}{\ell_n- \sum \limits_{j \in \textbf{V}_{l-1} \cup \tilde{\mathcal{V}}_i}w_{j}}\right].
\end{aligned}
\end{equation*}
Let 
$$E_{l}= \left\{\mathbbm{1}(K(i,\epsilon'))\sum_{j \in \textbf{V}_{l-1} \cup \tilde{\mathcal{V}}_i}^{l}w_{j} < 2l^{1/2}\ell_n^{1/2}\right\},$$
then:
\begin{equation}
\label{summ_eps}
    \mathbb{E}[w_{\tilde{v}(l)}\mathbbm{1}(K(i,\epsilon'))] =  \mathbb{E}\left[w_{\tilde{v}(l)}\mathbbm{1}(K(i,\epsilon'))|E_{l}\right] \mathbb{P}\left(E_{l}\right)+\mathbb{E}\left[w_{\tilde{v}(l)}\mathbbm{1}(K(i,\epsilon'))|\bar{E}_{l}\right] \mathbb{P}\left(\bar{E}_{l}\right).
\end{equation}
By Lemma \ref{sum_capacities}, and the fact that $l = o(n)$:
\begin{equation*}
\begin{aligned}
\label{nimp}
\mathbb{P}(\bar{E}_{l}) &\leq \exp\left(\frac{-l^{1/2}n^{1/6}}{4}\right) \\
&\leq \exp\left(\frac{-n^{1/6}}{4}\right).
\end{aligned}
\end{equation*}
Hence, by Conditions \ref{cond_nodess} it follows that:
$$
\mathbb{E}\left[ w_{\tilde{v}(l)}\mathbbm{1}(K(i,\epsilon')) |\bar{E}_{l}\right] \mathbb{P}\left(\bar{E}_{l}\right) \leq n^{1/3}\exp\left(\frac{-n^{1/6}}{4}\right)=o\left(\frac{1}{n}\right),
$$
together with Equation \eqref{summ_eps} this yields:
\begin{equation}
\begin{aligned}
\label{summ_eps_2}
\mathbb{E}[w_{\tilde{v}(l)}\mathbbm{1}(K(i,\epsilon'))] &=  \mathbb{E}\left[w_{v(l)} |E_{l}\right]\left(1+o\left(\frac{1}{n}\right)\right)+o\left(\frac{1}{n}\right).
\end{aligned}
\end{equation}
Moreover, by definition of the event $E_l$:
\begin{equation*}
\begin{aligned}
\mathbb{E}\left[w_{\tilde{v}(l)}\mathbbm{1}(K(i,\epsilon')) | E_{l}\right] &= \mathbb{E}\left[\mathbbm{1}(K(i,\epsilon'))\left.\sum_{j \not\in \textbf{V}_{l-1} \cup \tilde{\mathcal{V}}_i} \frac{w_j^2}{\ell_n\left(1-\frac{\sum_{j \in \textbf{V}_{l-1} \cup \tilde{\mathcal{V}}_i}w_{j}}{\ell_n}\right)} \right|E_{l} \right]  \\
&=  \mathbb{E}\left[\mathbbm{1}(K(i,\epsilon'))\left.\sum_{j \not\in \textbf{V}_{l-1} \cup \tilde{\mathcal{V}}_i} \frac{w_j^2}{\ell_n}\left(1+\frac{\sum_{j \in \textbf{V}_{l-1} \cup \tilde{\mathcal{V}}_i}w_{j}}{\ell_n}\right) \right| E_{l}\right] + o\left(\frac{l}{n}\right).
\end{aligned}
\end{equation*}
Let: 
$$
A = \mathbb{E}\left[\frac{\mathbbm{1}(K(i,\epsilon'))(\sum_{j \in \textbf{V}_{l-1} \cup \tilde{\mathcal{V}}_i}w_j)(\sum_{i = 1}^nw_j^2)}{\ell_n^2} \middle| E_{l} \right],
$$
and
$$
B = \mathbb{E}\left[\frac{\mathbbm{1}(K(i,\epsilon'))\sum_{j \in \textbf{V}_{l-1} \cup \tilde{\mathcal{V}}_i}w_j^2}{\ell_n} \middle| E_{l}\right].
$$
recall the definition of $t_1$ and $t_2$ from the chapter before. By Lemmas \ref{mean_poww_capacities_2} and  \ref{mean_ccapacities}, and by definition of $E_l$ and $\mathbbm{1}(K(i,\epsilon'))$ and Conditions \ref{cond_nodess}:
\begin{equation*}
\begin{aligned}
    A-B &= \mathbb{E}\left[\frac{\mathbbm{1}(K(i,\epsilon'))(\sum_{j \in \textbf{V}_{l-1} \cup \tilde{\mathcal{V}}_i}(w_j-w_j^2))}{\ell_n} \middle| E_{l} \right] + o\left(\frac{n^{2/3}}{n}\right) \\
    &\leq \mathbb{E}\left[\frac{\sum_{j \in \textbf{V}_{l-1} }(w_j-w_j^2)+\sum\limits_{j = t_1}^{t_2}\left(w_{v(j)}-w_{v(j)}^2\right)}{\ell_n} \middle| E_{l} \right] + o\left(\frac{n^{2/3}}{n}\right) \\
    &\leq \frac{l+(t_2-t_1)}{\ell_n}\left(1-C\right) + o\left(\frac{n^{2/3}}{n}\right) \\
    &\leq \frac{l+2(1-\epsilon')f(i)\ell_n^{2/3}}{\ell_n}\left(1-C\right) + o\left(\frac{n^{2/3}}{n}\right).
\end{aligned}
\end{equation*}
It follows that:
\begin{equation*}
\begin{aligned}
\mathbb{E}\left[w_{\tilde{v}(l)}\mathbbm{1}(K(i,\epsilon')) | E_{l}\right] &= \mathbb{E}\left[\mathbbm{1}(K(i,\epsilon'))\left.\sum_{j \not\in \textbf{V}_{l-1} \cup \tilde{\mathcal{V}}_i} \frac{w_j^2}{\ell_n}\left(1+\frac{\sum_{j \in \textbf{V}_{l-1} \cup \tilde{\mathcal{V}}_i}w_{j}}{\ell_n}\right) \right| E_{l}\right] + o\left(\frac{l}{n}\right) \\
&= \frac{\sum_{i = 1}^nw_j^2}{\ell_n} + A - B \\
&- \mathbb{E}\left[\mathbbm{1}(K(i,\epsilon'))\frac{\left(\sum_{j \in \textbf{V}_{l-1} \cup \tilde{\mathcal{V}}_i}w_j^2\right)\left(\sum_{j \in \textbf{V}_{l-1} \cup \tilde{\mathcal{V}}_i}w_j\right)}{\ell_n^2} \middle| E_{l}\right]
+  o\left(\frac{l}{n}\right) \\
&= \frac{\sum_{i = 1}^nw_j^2}{\ell_n} + A - B + o\left(\mathbb{E}\left[\mathbbm{1}(K(i,\epsilon'))\frac{\sum_{j \in \textbf{V}_{l-1} \cup \tilde{\mathcal{V}}_i}w_j^2}{\ell_n} \middle| E_{l}\right]\right)\\ &+o\left(\frac{l+f(i)n^{2/3}}{n}\right) \\
&\leq 1 +\frac{l+2(1-\epsilon')f(i)\ell_n^{2/3}}{\ell_n}\left(1-C\right)+o\left(\frac{l+f(i)n^{2/3}}{n}\right).
\end{aligned}
\end{equation*}
Replacing in Equation \eqref{summ_eps_2} finishes the proof.
\end{proof}
We draw the attention of the reader to the fact that this proof is very similar to the one already done in Section $2$ of  \cite{O20}. Since, conditionally on  $ \tilde{\mathcal{V}}_i$ , the graph $G(\mathcal{V} \setminus \tilde{\mathcal{V}}_i,p_{f(i+1)})$ is an inhomogeneous random graph. This calculation shows that, when $K(i,\epsilon')$ holds, the exploration process of $G(\mathcal{V} \setminus \tilde{\mathcal{V}}_i,p_{f(i+1)})$ has slightly smaller than $1$ increments on expectation.  Thus,  when $K(i,\epsilon')$ holds, $G(\mathcal{V} \setminus \tilde{\mathcal{V}}_i,p_{f(i+1)})$ is slightly subcritical. In  \cite{O20}, we study extensively the graph $G(\mathcal{V} \setminus \tilde{\mathcal{V}}_i,p_{f(i)})$. When the event  $K(i,\epsilon')$ holds those two graphs are both slightly subcritical and have very similar increments distribution. Hence, the upper bound results that where shown for $G(\mathcal{V} \setminus \tilde{\mathcal{V}}_{i},p_{f(i)})$ can also be shown for $G(\mathcal{V} \setminus \tilde{\mathcal{V}}_{i},p_{f(i+1)})$ with slightly larger constants  by following the same method of proofs from  \cite{O20} under the event $K(i,\epsilon')$ and then adding the probability of $K(i,\epsilon')$ not holding. By this argument the following two theorems are true similarly to their counterparts for the graph $G(\mathcal{V} \setminus \tilde{\mathcal{V}}_{i},p_{f(i)})$. Theorem \ref{Theo11} is related to Theorem $3$ from  \cite{O20}, and Theorem \ref{Theo10} is related to Theorem $4$ from  \cite{O20}. We thus omit their proofs.
\begin{Theorem}
\label{Theo11}
Let $1 \geq \epsilon > 0$ and $\epsilon'>0$ and consider the following event: \\
All the connected components of $G(\mathcal{V} \setminus \tilde{\mathcal{V}}_i, p_{f(i+1)})$ have size smaller than
$$
\frac{\ell_n^{2/3}}{Cf(i)^{1-\epsilon}},
$$
and weight smaller than
$$
\frac{(1+\epsilon')\ell_n^{2/3}}{Cf(i)^{1-\epsilon}}.
$$
There exists a positive constant $A >0$ that only depend on the distribution of $W$ such that the probability of this event not happening is at most:
$$
A\left(\exp\left(\frac{-f(i+1)^{\epsilon}}{A}\right)+\exp\left(\frac{-\sqrt{f(i+1)}}{A}\right)+\exp\left(\frac{-n^{1/12}}{A}\right)\right).
$$
\end{Theorem}
\begin{Theorem}
\label{Theo10}
Let $\textnormal{Exc}$ be the maximal excess of the connected components of $G(\mathcal{V} \setminus \tilde{\mathcal{V}}_i, p_{f(i+1)})$. There exists a positive constant $A >0$ such that, for any $1 \geq \epsilon > 0$ the probability of $$\textnormal{Exc} \geq Af(i+1)^{\epsilon},$$ is at most:
$$ A\left(\exp\left(\frac{-f(i+1)^{\epsilon/2}}{A}\right)+\exp\left(\frac{-\sqrt{f(i+1)}}{A}\right)+\exp\left(\frac{-n^{1/12}}{A}\right)\right).
$$
\end{Theorem}
Those theorems deal with the sizes, weights and excesses of the connected components of $G(\mathcal{V} \setminus \tilde{\mathcal{V}}_i, p_{f(i+1)})$. In order to prove bounds on the diameter of those components, we still need to bound the height of some well-chosen spanning trees and use Lemma \ref{excess_height}. This is why we bound the height of their exploration trees.
\subsection{Coupling with Galton-Watson trees}
\label{section5}
In this part, $i \geq 1$ is fixed. The construction described bellow is done conditionally on $\tilde{\mathcal{V}}_i$, the vertices set of the largest connected component of $G(\textbf{W},p_{f(i)})$.\par
Let $\tilde{L}$ be the exploration process of $G(\mathcal{V} \setminus \tilde{\mathcal{V}}_i, p_{f(i+1)})$. And for $k \geq 1$, let $\tilde{c}(k)$ be the number of children of node $\tilde{v}(k)$ in the exploration process. If a connected component of $G(\mathcal{V} \setminus \tilde{\mathcal{V}}_i, p_{f(i+1)})$ is discovered at time $k$ then its exploration process will be:
$$
\tilde{L}'_k(0)=1.
$$
And for $i \geq 1$:
$$
\tilde{L}'_k(i+1) = \tilde{L}'_k(i)+\tilde{c}(k+i-1)-1.
$$
This exploration process stops when it hits $0$. 
For $(k,l) \in (\mathcal{V} \setminus \tilde{\mathcal{V}}_i)^2$, let $X(k,l,i+1)$ be a Bernoulli random variable of parameter 
$$1-\exp(-w_{k}w_{l}p_{f(i+1)}).$$
Then for any $k \geq 1$, by definition of the exploration process, if a connected component is discovered at time $k$ the number of children of node $\tilde{v}(k)$ has the the same distribution as:
$$
\sum_{\tilde{v}(l) \in \mathcal{V} \setminus \tilde{\mathcal{V}}_i, \, l \geq 1+k} X(\tilde{v}(k),\tilde{v}(l),i+1).
$$
Generally, if node $\tilde{v}(k+i)$ is in the same connected component as $\tilde{v}(k)$, then the distribution of $\tilde{c}(k+i)$  will be the same as the distribution of:
$$
\sum_{l \geq \tilde{L}'_k(i)+k+i} X(\tilde{v}(k+i),\tilde{v}(l),i+1).
$$
For each $k \geq 0$, let $\tilde{\mathbb{T}}(k)$ be the tree with exploration process $\tilde{L}'_k$. If we can bound the height of $\tilde{\mathbb{T}}(k)$  for every $k \geq 1$, then we will have a bound on the height of the largest exploration tree in $G(\mathcal{V} \setminus \tilde{\mathcal{V}}_i, p_{f(i+1)})$. 
In order to do so, we create  a Galton-Watson tree $\mathbb{T}(k)$ that contains a sub-tree with the same distribution as $\tilde{\mathbb{T}}(k)$ in the following way: \par
Conditionally on $\textbf{V}_{k-1} = (\tilde{v}(1), \tilde{v}(2), ..., \tilde{v}(k-1))$ and $\tilde{\mathcal{V}}_i$, we construct a Galton-Watson tree $\mathbb{T}(k)$. Start by one node, and for each new node created follow these two steps. First give the node a label which is a random variable with the distribution of $\tilde{v}(k)$ independently of everything else. And then conditionally on the node  label being $r$ give it a random number of children which is a Poisson random variable of parameter:
$$w_{r}\left(\ell_n-\sum_{l \in \textbf{V}_{k-1}\cup \tilde{\mathcal{V}}_i}w_l\right)p_{f(i+1)},$$
independently of everything else \footnote{Here we abused notations by using the $\textbf{V}_{k-1}$ as a set and as a sequence at the same time. It remains clear from the context which concept we are using.}. 
Remark that this process may not end, in that case we get an infinite tree.  $\mathbb{T}(k)$ is thus a Galton-Watson tree with reproduction distribution $\mu(k,i+1)$, which is, conditionally on $(\textbf{V}_k,\tilde{\mathcal{V}}_i) $, a Poisson distribution of parameter $$w_{\tilde{v}(k)}\left(\ell_n-\sum_{l \in \textbf{V}_{k-1}\cup \tilde{\mathcal{V}}_i}w_l\right)p_{f(i+1)}.$$
In order to have a clear order on the nodes of $\mathbb{T}(k)$, we sort them in a breadth-first search fashion while creating the tree.
We start from one node, which is the root of $\mathbb{T}(k)$. We give it order $1$, label $r(1)$. The root then has $b(1)$ children with the distributions described above. 
The children of the root are given orders in $(2,3...,b(1)+1)$ randomly uniformly. Now do  same for the the node with order $2$ and order its children $(b(1)+2,b(1)+2,...,b(1)+b(2)+1)$ randomly uniformly too. Continue in this fashion until there are no new nodes or the construction never ends. We call this the breadth-first order of $\mathbb{T}(k)$ . \par
We construct a subtree $\mathbb{T}'(k)$ by pruning $\mathbb{T}(k)$ with the following algorithm.  $\mathbb{L}$ is the list of nodes currently investigated by the algorithm. Initially $\mathbb{L}$ contains just the root : 
\begin{enumerate}
    \item At step $1$ color the root in red, then add the children of the root to $\mathbb{L}$.
    \item At step $i$ consider the $i$-th node added to $\mathbb{L}$ following the breadth-first order. If there are no other red nodes with the same label yet, then color it in red and add its childer to $\mathbb{L}$.
    \item Continue like this as long as $\mathbb{L}$ is not empty.
\end{enumerate}
\begin{figure}[!htbp]
\centering
\includegraphics[width=1.\textwidth]{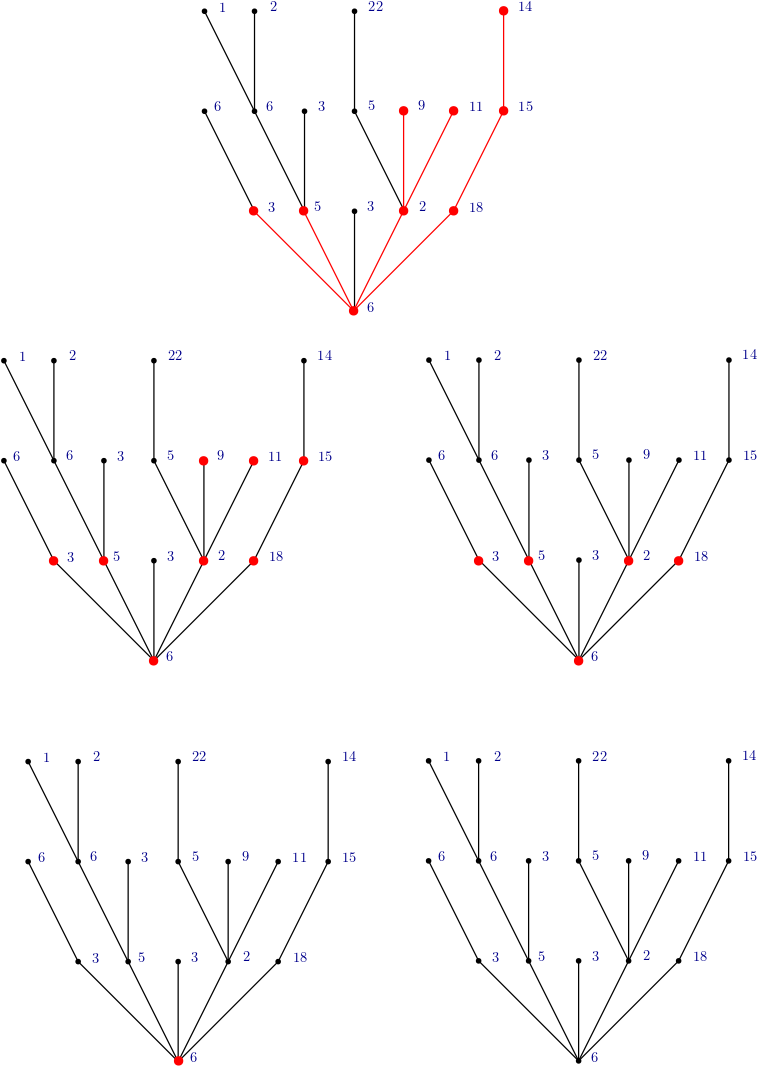}
\caption{An example of the colouring procedure. The breath-first order is implicitly induced by the drawing, nodes are ordered from bottom to top and from left to right. Each node has its label directly to its right. We only show the steps where a whole generation has been considered.}
\end{figure}
Since there is only a finite number of labels, this process will eventually end. The nodes are always removed from $\mathbb{L}$ following the breadth-first order. After finishing this procedure, we obtain a subtree of $\mathbb{T}(k)$ composed only of red nodes that we denote by $\mathbb{T}'(k)$. This fact is easily seen from the procedure since a node can only be red if its parent is also red. Moreover, we have the following lemma. This result was already shown in \cite{no06} (proposition $3.1$) and proven again in \cite{SRJ10} for $\mathbb{T}'(1)$. The proof for general $\mathbb{T}'(k)$ is the same conditionally on $\textbf{V}_{k-1}$, and so we refer to those articles for a proof.
\begin{Lemma}
\label{lem9}
For $k \geq 1$, conditionally on $(\tilde{v}(1), \tilde{v}(2), ..., \tilde{v}(k-1))$, the tree $\mathbb{T}'(k)$  has the same distribution as $\tilde{\mathbb{T}}(k)$.
\end{Lemma}

By Lemma \ref{lem9}, in order to bound the heights of the exploration trees of $G(\mathcal{V} \setminus \tilde{\mathcal{V}}_i, p_{f(i+1)})$ it is sufficient to bound the heights of the trees $\mathbb{T}'(k)$ for $k \geq 1$. Since the  $\mathbb{T}'(k)$'s are substrees of the $\mathbb{T}(k)$'s, we will bound the heights of the later. The following lemma revolves around classical results on Galton-Watson trees. 
\begin{Lemma}
\label{lem12}
For $k \geq 1$. Let $H(\mathbb{T}(k))$ be the height of tree $\mathbb{T}(k)$. For $m \geq 0$, denote the event $\left\{H(\mathbb{T}(k)) \geq m\right\} \cap  K(i,1/8)$ by $\textbf{Q}(k,m)$. If  $ k \leq \ell_n^{6/7}$ then:
$$
\mathbb{P}\left(\textbf{Q}\left(k,\frac{2\ell_n^{1/3}}{\sqrt{f(i+1)}}\right)\right) \leq  \frac{A\sqrt{f(i+1)}}{{\ell_n^{1/3}} }\exp\left(\frac{-Ck}{2\sqrt{f(i+1)}\ell_n^{2/3}} - \frac{\sqrt{f(i+1)}}{8}\right) .
$$
And if $k \geq \ell_n^{6/7}$ then:
\begin{equation*}
\begin{aligned}
\mathbb{P}\left(\textbf{Q}\left(k,\frac{2\ell_n^{1/3}}{\sqrt{f(i+1)}}\right)\right) \leq \frac{A\sqrt{f(i+1)}}{{\ell_n^{1/3}} }\exp\left(\frac{-C\ell_n^{6/7}}{2\sqrt{f(i+1)}\ell_n^{2/3}} - \frac{\sqrt{f(i+1)}}{8}\right).
\end{aligned}
\end{equation*}
\end{Lemma}
\begin{proof}
For simplicity and without loss of generality suppose that $m$ is even. Clearly:
\begin{equation}
\begin{aligned}
\label{eqf11}
\mathbb{P}(\textbf{Q}(k,m)) &\leq \mathbb{E}[\mu(k,i+1)\mathbbm{1}(K(i,1/8))]\mathbb{P}(\textbf{Q}(k,m-1)))\\
&\leq \mathbb{E}[\mu(k,i+1)\mathbbm{1}(K(i,1/8))]^{m/2}\mathbb{P}(\textbf{Q}(k,m/2))
\end{aligned}
\end{equation}
Moreover, $\mathbbm{1}(K(i,1/8))$ is measurable with regard to $\tilde{\mathcal{V}}_i$ by definition. This remark and Lemma \ref{mean_corrr} alongside the fact that for $x \in \mathbb{R}$, $1-e^{-x} \leq x$ yield:
\begin{equation*}
    \begin{aligned}
    \mathbb{E}[\mu(v(k),i+1)\mathbbm{1}(K(i,1/8))] &= \mathbb{E}\left[\mathbb{E}\left[\mu(v(k),i+1)\mathbbm{1}(K(i,1/8))\middle| \tilde{\mathcal{V}}_i, \textbf{V}_k\right]\right] \\
    &\leq \mathbb{E}\left[w_{\tilde{v}(k)}\left(\ell_n-\sum_{l \in \textbf{V}_{k-1}\cup \tilde{\mathcal{V}}_i}w_l\right)p_{f(i+1)} \mathbbm{1}(K(i,1/8))\right] \\
    &\leq \mathbb{E}\left[w_{\tilde{v}(k)}\left(\ell_n-\frac{7f(i)\ell_n^{2/3}}{4C}\right)p_{f(i+1)}\mathbbm{1}(K(i,1/8))\right]-\frac{k}{\ell_n}(1+o(1)).
    \end{aligned}
\end{equation*}
Hence, by Lemma \ref{esperances}, a quick calculation yields for $k \leq \ell_n^{6/7}$:
\begin{equation}
\begin{aligned}
\label{eqf12}
&\mathbb{E}[\mu(k,i+1)\mathbbm{1}(K(i,1/8))] \\
&\leq \left(1+\frac{k+\frac{7}{4C}f(i)\ell_n^{2/3}}{\ell_n}\left(1-C\right)\right)\left(1+\frac{f(i+1)}{\ell_n^{1/3}}-\frac{7f(i)}{4C\ell_n^{1/3}}\right) - \frac{k}{\ell_n}+ o\left(\frac{k+f(i)n^{2/3}}{n}\right) \\
&\leq 1-\frac{Ck}{\ell_n}-\frac{f(i)}{4\ell_n^{1/3}}+o\left(\frac{k+f(i)n^{2/3}}{n}\right)
\end{aligned}
\end{equation}
It is also well known, for critical Galton-Watson trees (see for example Lemma $6.7$ of \citet*{RB18}), that there exists a constant $A>0$ such that for any $m\geq 0$:
\begin{equation}
\label{eq13}
\mathbb{P}(\textbf{Q}(k,m/2)) \leq \frac{A}{m},
\end{equation}
this result is thus also true for a subcritical Galton-Watson tree. \par
Hence, by Equation \eqref{eqf11}, \eqref{eqf12}, and \eqref{eq13}, a simple computation yields:
\begin{equation}
\begin{aligned}
\label{eq14}
&\mathbb{P}\left(\textbf{Q}\left(k,\frac{2\ell_n^{1/3}}{\sqrt{f(i+1)}}\right)\right) \leq \frac{A\sqrt{f(i+1)}}{{\ell_n^{1/3}} }\mathbb{E}[\mu(k,i+1)\mathbbm{1}(K(i,1/8))] ^{\frac{\ell_n^{1/3}}{\sqrt{f(i+1)}}}\\
&=\frac{A\sqrt{f(i+1)}}{{\ell_n^{1/3}} }\exp\left(\frac{-Ck}{\sqrt{f(i+1)}\ell_n^{2/3}} - \frac{\sqrt{f(i+1)}}{4}+o\left(\frac{k+f(i)n^{2/3}}{\sqrt{f(i+1)}n^{2/3}}\right)\right)\\
&\leq \frac{A\sqrt{f(i+1)}}{{\ell_n^{1/3}} }\exp\left(\frac{-Ck}{2\sqrt{f(i+1)}\ell_n^{2/3}} - \frac{\sqrt{f(i+1)}}{8}\right) 
\end{aligned}
\end{equation}

Moreover, for $k \geq \ell_n^{6/7}$, by Lemma \ref{decreasingg}:
$$
\mathbb{E}[w_{\tilde{v}(k)}\mathbbm{1}(K(i,1/8))] \leq \mathbb{E}\left[w_{\tilde{v}\left(\ell_n^{6/7}\right)}\mathbbm{1}(K(i,4))\right].
$$
With this fact and Equations \eqref{eq14}, we get for $k \geq \ell_n^{6/7}$:
\begin{equation*}
\begin{aligned}
\mathbb{P}\left(\textbf{Q}\left(k,\frac{2\ell_n^{1/3}}{\sqrt{f(i+1)}},\right) \geq \frac{2\ell_n^{1/3}}{\sqrt{f(i+1)}}\right) \leq \frac{A\sqrt{f(i+1)}}{{\ell_n^{1/3}} }\exp\left(\frac{-C\ell_n^{6/7}}{2\sqrt{f(i+1)}\ell_n^{2/3}} - \frac{\sqrt{f(i+1)}}{8}\right)
\end{aligned}
\end{equation*}
\end{proof}
In order to bound the height of all the exploration trees at once, we need a bound on the number of connected components in $G(\mathcal{V} \setminus \tilde{\mathcal{V}}_i, p_{f(i+1)})$. The order of nodes in the exploration process of $G(\mathcal{V} \setminus \tilde{\mathcal{V}}_i, p_{f(i+1)})$ has the same distribution as that of $G(\mathcal{V} \setminus \tilde{\mathcal{V}}_i, p_{f(i)})$. Moreover conditionally on that order, the exploration process of $G(\mathcal{V} \setminus \tilde{\mathcal{V}}_i, p_{f(i)})$ is stochastically upper bounded by the exploration process of $G(\mathcal{V} \setminus \tilde{\mathcal{V}}_i, p_{f(i+1)})$. Since, by definiton, the number of connected components discovered in an interval of time of the BFW correspond to the number of new minimums attained in that same interval. It is clear by this argument that for any $m \in \mathbb{N}$ and $(t_1,t_2) \in \mathbb{N}^2$ such that $ t_1 \leq t_2$, the probability of discovering more than  $m$ connected components in the BFW  of $G(\mathcal{V} \setminus \tilde{\mathcal{V}}_i, p_{f(i+1)})$ between times $t_1$ and $t_2$ is smaller than the probability of discovering more than $m$ connected components in the BFW of $G(\mathcal{V} \setminus \tilde{\mathcal{V}}_i, p_{f(i)})$.

\begin{Lemma}
\label{lem13}
Let $r_1=\frac{\ell_n^{2/3}}{\sqrt{f}C}$.
Let $D_0$ be the event "The number of connected component discovered before time $r_1$ in the exploration process of $G(\mathcal{V} \setminus \tilde{\mathcal{V}}_i, p_{f(i+1)})$ is larger than $\frac{\sqrt{f(i)}\ell_n^{1/3}}{2C}$." \par
For $j \geq 1$, consider the interval 
$$I_j=\left[r_1+\frac{(j^2-1)f(i+1)\ell_n^{2/3}}{C} ,r_1+\frac{((j+1)^2-1)f(i+1)\ell_n^{2/3}}{C}\right)$$
and let $\tilde{j}$ be the greatest integer such that:
$$\frac{(\tilde{j}^2-1)f(i+1)\ell_n^{2/3}}{C} \leq \ell_n^{6/7}.$$
For some $\tilde{j} \geq j \geq 1$, let $D_j$ be the event "There are more than $j^3f(i)\ell_n^{1/3}$ connected components discovered in the interval $I_j$ in the exploration process of $G(\mathcal{V} \setminus \tilde{\mathcal{V}}_i, p_{f(i+1)})$." \par
There exists a constant $A >0$ such that:
$$
\mathbb{P}\left(\bigcup_{j=0}^{\tilde{j}}(D_j)\right) \leq A\exp\left(\frac{-\sqrt{f}}{A}\right)+A\exp\left(\frac{-n^{1/12}}{A}\right).
$$
\end{Lemma}
\begin{proof}
By the arguments given before, it is sufficient to show this result for the exploration process of $G(\mathcal{V} \setminus \tilde{\mathcal{V}}_i, p_{f(i)})$.
Recall that the exploration process can only go down by $1$, so the number of connected components  discovered in an interval of time is equal the number of times a new minimum was attained by the exploration process in that same interval of time. Hence, by Equation $75$ in  \cite{O20}, the probability of discovering more than $\frac{\sqrt{f}\ell_n^{1/3}}{2C}$ connected components before time $r_1$ in the exploration process of $G(\mathcal{V} \setminus \tilde{\mathcal{V}}_i, p_{f(i)})$ is at most:
\begin{equation}
\label{eqq2}
 A\exp\left(\frac{-\sqrt{f}}{A}\right).
\end{equation}
And by using the union bound between Theorem $26$ and Corollary $36.2$ in  \cite{O20} we obtain:
\begin{equation}
\label{eqq1}
\mathbb{P}\left(\bigcup_{j=1}^{\tilde{j}}(D_j)\right) \leq A\exp\left(\frac{-\sqrt{f}}{A}\right)+A\exp\left(\frac{-n^{1/12}}{A}\right).
\end{equation}
\end{proof}

With Lemmas \ref{lem12} and \ref{lem13} we get bounds on the height of the exploration trees of $G(\mathcal{V} \setminus \tilde{\mathcal{V}}_i, p_{f(i+1)})$. On the other hand, Theorem \ref{Theo10} gives bounds on its excess.
With this in hand, the following theorem combines those results in order to give an upper bound for the diameter of $G(\mathcal{V} \setminus \tilde{\mathcal{V}}_i, p_{f(i+1)})$.
\begin{Theorem}
\label{them15}
There exists a constant $A > 0$ such that:
\begin{equation*}
\mathbb{P}\left(\textnormal{diam}(G(\mathcal{V} \setminus \tilde{\mathcal{V}}_i, p_{f(i+1)})) \geq \frac{A\ell_n^{1/3}}{f(i+1)^{1/4}}\right) \leq A\left(\exp\left(\frac{-f(i+1)^{1/8}}{A}\right)\right).
\end{equation*}
\end{Theorem}
\begin{proof}
By Equation \eqref{eq8}, there exists a constant $A >0$ such that, the probability of the event  $K\left(i,1/8\right)$ not holding is at most:
\begin{equation}
\label{eq141}
    A\exp\left(\frac{-f(i)}{A}\right).
\end{equation}
When that event holds, it is to sufficient study the event $\textbf{Q}\left(k,\frac{2\ell_n^{1/3}}{\sqrt{f(i+1)}}\right)$ and bound the surplus and then use the union bound. Recall the definition of the events $(D_j)_{j \leq \tilde{j}}$ and the intervals $(I_j)_{j \leq \tilde{j}}$ from Lemma \ref{lem13}. If all those events are verified, then by the union bound, in each interval $I_j$, the  probability that there exists an exploration tree with height larger than $\frac{2\ell_n^{1/3}}{\sqrt{f(i+1)}}$ is at most:
$$
B_j = j^3f(i)\ell_n^{1/3} \frac{A\sqrt{f(i+1)}}{{\ell_n^{1/3}} }\exp\left(\frac{-C(j^2-1)f(i+1)\ell_n^{2/3}}{2C\sqrt{f(i+1)}\ell_n^{2/3}} - \frac{\sqrt{f(i+1)}}{8}\right).
$$
Recall that $r_1=\frac{\ell_n^{2/3}}{\sqrt{f}C}$. The probability that there exists an exploration tree with height larger than $\frac{2\ell_n^{1/3}}{\sqrt{f(i+1)}}$ before time $r_1$ is at most:
$$
B_0 = \frac{\sqrt{f(i)}\ell_n^{1/3}}{2C}\frac{A\sqrt{f(i+1)}}{{\ell_n^{1/3}} }\exp\left(- \frac{\sqrt{f(i+1)}}{8}\right).
$$
Hence, for it to be an exploration tree with height larger than $\frac{2\ell_n^{1/3}}{\sqrt{f(i+1)}}$, either one of the events $(D_j)_{0 \leq j \leq \tilde{j}}$ does not hold or $K\left(i,1/8\right)$ does not hold, or one of the event $\textbf{Q}\left(k,\frac{2\ell_n^{1/3}}{\sqrt{f(i+1)}}\right)$ does not hold or there is a connected component discovered after time $\ell_n^{6/7}$ with an exploration tree with height larger than $\frac{2\ell_n^{1/3}}{\sqrt{f(i+1)}}$. There are less than $n$ connected component discovered after time  $\ell_n^{6/7}$.
Hence, by Equation \eqref{eq141} and Lemmas \ref{lem12} and \ref{lem13}, there exists constants $A >0$ and $A'>0$ such that, the probability that there exists a connected component with an exploration tree of height larger than $\frac{2\ell_n^{1/3}}{\sqrt{f(i+1)}}$ is at most:
\begin{equation}
\begin{aligned}
\label{eqf142}
&A\exp\left(\frac{-f(i)}{A}\right)+\sum_{j=0}^{\tilde{j}} B_j + n\frac{A\sqrt{f(i+1)}}{{\ell_n^{1/3}} }\exp\left(\frac{-C\ell_n^{6/7}}{2\sqrt{f(i+1)}\ell_n^{2/3}} - \frac{\sqrt{f(i+1)}}{8}\right)+\mathbb{P}\left(\bigcup_{j=1}^{\tilde{j}}(D_j)\right) \\
&\leq A'\exp\left(\frac{-\sqrt{f(i)}}{A'}\right)+A'\exp\left(\frac{-n^{1/12}}{A'}\right).
\end{aligned}
\end{equation}
Recall that $\textnormal{Exc}$ is the maximal excess of the connected components of $G(\mathcal{V} \setminus \tilde{\mathcal{V}}_i, p_{f(i+1)})$. By Theorem \ref{Theo10}, there exists a positive constant $A'' >0$ such that, for any $1 > \epsilon > 0$:
\begin{equation}
\begin{aligned}
\label{eq143}
\mathbb{P}\left(\textnormal{Exc} \geq A''f(i+1)^{\epsilon}\right) \leq &A''\left(\exp\left(\frac{-f(i+1)^{\epsilon/2}}{A''}\right)+\exp\left(\frac{-\sqrt{f(i+1)}}{A''}\right)+\exp\left(\frac{-n^{1/12}}{A''}\right)\right).
\end{aligned}
\end{equation}
Taking $\epsilon = 1/4$ in Equation \eqref{eq143} and using Lemma \ref{excess_height} with Equations \eqref{eqf142} and \eqref{eq143} then yields:
\begin{equation*}
\begin{aligned}
\mathbb{P}\left(\textnormal{diam}(G(\mathcal{V} \setminus \tilde{\mathcal{V}}_i, p_{f(i+1)})) \geq \frac{A\ell_n^{1/3}}{f(i+1)^{1/4}}\right) &\leq A\left(\exp\left(\frac{-f(i+1)^{1/8}}{A}\right)
+\exp\left(\frac{-n^{1/12}}{A}\right)\right) \\
&\leq A'\left(\exp\left(\frac{-f(i+1)^{1/8}}{A'}\right)\right).
\end{aligned}
\end{equation*}
\end{proof}

\section{Bounding the diameter of the giant component}
\label{sec4}

\subsection{A new pruning procedure for the giant component}

Take $f \geq f(0)$, we consider the exploration process until time $t = \frac{5f\ell_n^{2/3}}{2C}$. We know by Theorem \ref{principal_a} that the giant component is fully discovered before $t$ with some high probability. In order to bound the diameter of the giant component, a natural idea would be to couple the trees discovered before times $t$ with a sequence of i.i.d. Galton-Watson trees with the distribution of $\tilde{\mathbb{T}}(1)$ by using the same colouring of the nodes described in Sub-section \ref{section5}. However,  $\tilde{\mathbb{T}}(1)$ is slightly supercritical. Thus, some of the Galton-Watson trees we couple with will have infinite size. To solve this problem we first add i.i.d. cuts with probability $\frac{2f}{\ell_n^{1/3}}$ to the edges of $G(\textbf{W},p_{f(i)})$. Then we couple the new smaller exploration trees obtained with i.i.d. Galton-Watson trees with the the distribution of:
$$
\sum_{l=1}^n \tilde{X}(v(1),l,i),
$$
where $\tilde{X}(k,l,i)$ is the product of two i.i.d. Bernoulli random variables. The first one corresponds to "having an edge" and has parameter: $$\left(1-\exp(-w_{k}w_{l}p_{f(i+1)})\right),$$
the second one corresponds to "not having a cut" and has parameter:
$$\left(1-\frac{2f}{\ell_n^{1/3}}\right).$$
We then apply the procedure of Sub-section \ref{section5} on those Galton-Watson trees.
This yields a new two steps procedure, first pruning then coloring, which gives a forest of i.i.d. slightly sub-critical Galton-Watson trees.
Moreover, the sum of heights of those Galton-Watson trees is an upper bound of the height of any exploration tree discovered before time $t$.\par
We start by bounding the number of Galton-Watson trees obtained by this construction. 
A new tree is created at each cut. This means that the total number of Galton-Watson trees after pruning is equal to the number of cuts plus the number of initial trees. However, if we are only interested in the height of one tree. Only the number of cuts is of interest. Hence, even though we have a bound on the initial number of trees before time $t$ by  \cite{O20}. We, in fact only need a bound on the number of cuts. This is given in the following lemma.
\begin{Lemma}
\label{lem23}
The number of cuts in the pruning procedure is less than $\frac{6f^2\ell_n^{1/3}}{C}$ with probability at least:
$$1-\exp\left(-2Cf^3\right).$$
\end{Lemma}
\begin{proof}
Because a tree of size $m$ has $m-1$ edges, there are less than $t$ edges in the exploration trees before time $t$. Let $J$ be the number of edge cuts, then:
$$
\mathbb{E}[J] \leq \frac{5f^2\ell_n^{1/3}}{C}.
$$
Hence, by Hoeffding's inequality:
\begin{equation*}
\begin{aligned}
\mathbb{P}\left(J \geq \frac{6f^2\ell_n^{1/3}}{C}\right) &\leq \exp\left(\frac{-2(f^2\ell_n^{1/3})^2}{t}\right) \\
&\leq \exp\left(\frac{-4Cf^3}{5}\right).
\end{aligned}
\end{equation*}
\end{proof}
Since $C = \frac{\mathbb{E}[W^3]}{\mathbb{E}[W]} \geq 1$, this lemma shows that w.h.p the number of Galton-Watson trees obtained by the pruning then coupling procedure is less than $6f^2\ell_n^{1/3}+1 \leq 7f^2\ell_n^{1/3}$. 
\subsection{Bounding the diameter}

With this in hand, we can prove a bound on the height of the exploration tree of the largest component. We denote the Galton-Watson trees constructed by pruning by $(\mathbb{T}^{\top}(k))_{k \geq 1}$. We use a simple union bound. Either there is some tree discovered before time $t$ with large height, or the largest connected component is discovered after time $t$. We already have a bound on the latter event, and we bound the former event to get the following theorem. Bounding the height of an exploration tree discovered before time $t$ by the sum of heights of all the $(\mathbb{T}^{\top}(k))_{k \geq 1}$'s will not yield a good enough upper bound. But this first step will allow us to use a recursive argument which yields the tight bound presented here.
\begin{Theorem}
\label{Theoo18}
There exist constants $A' > 0$ and $A > 0$ such that the exploration tree of the largest component of $G(\textbf{W},p_{f})$, denoted by $H(\textbf{W},p_f)$, has height smaller than $A'f^2\ell_n^{1/3}$ with probability at least:
$$
1-A\exp\left(\frac{-f}{A}\right).
$$
\end{Theorem}

\begin{proof}
Let $\mathcal{B}_1$ be the event:  the exploration of the largest component of $G(\textbf{W},p_{f})$ is not already done at time $t$. By Theorem $1$ in  \cite{O20}:
\begin{equation}
\label{eql1}
\mathbb{P}(\mathcal{B}_1) \leq A\exp\left(\frac{-f}{A}\right),
\end{equation}
with $ A  >0$ a large constant. \par
Let $R$ be the random number of cuts obtained after pruning the exploration trees up to time $t$. And let $\mathcal{B}_2$ be the event: $R$ is larger than  $7f^2\ell_n^{1/3}$. By Lemma \ref{lem23}:
\begin{equation}
\label{eql2}
\mathbb{P}(\mathcal{B}_2) \leq A\exp\left(\frac{-f^3}{A}\right).
\end{equation}
If the largest component is discovered before time $t$, the  height of its exploration tree $\text{ht}(H(\textbf{W},p_f))$ will be smaller than the sum of heights of all the trees obtained after pruning. By exactly the same method we used in the proof of Lemma \ref{lem12}, there exists $A >0$ such that for any $r >0$:
\begin{equation}
\label{eqf122}
\mathbb{P}(\text{ht}(\mathbb{T}^{\top}(1))) \geq r) \leq \frac{A}{r}\exp\left(\frac{-fr}{2\ell_n^{1/3}}\right),
\end{equation}
and by the union bound:
\begin{equation}
\label{eql3}
\mathbb{P}\left(\sup_{k \leq 7f^2\ell_n^{1/3}} \text{ht}(\mathbb{T}^{\top}(k))) \geq \ell_n^{1/3}\right) \leq A'\exp\left(\frac{-f}{A'}\right),
\end{equation}
with $A'>0$ a large constant.
Denote the event $$\left\{\sup_{k \leq 7f^2\ell_n^{1/3}} \text{ht}(\mathbb{T}^{\top}(k))) \geq \ell_n^{1/3}\right\}$$
by $\mathcal{B}_3$. Using Equations \eqref{eql1} and \eqref{eql2} alongside Equation \eqref{eql3} shows that:
\begin{equation}
\label{eql4}
\mathbb{P}(\mathcal{B}_1\cup \mathcal{B}_2 \cup \mathcal{B}_3) \leq A\exp\left(\frac{-f}{A}\right).
\end{equation}
Let $\tilde{\mathcal{B}} = \bar{\mathcal{B}}_1 \cap \bar{\mathcal{B}}_2 \cap \bar{\mathcal{B}}_3$. If $\tilde{\mathcal{B}}$ holds, then the height of the exploration tree of the largest component is smaller than the sum of heights of the trees $((\mathbb{T}^{\top}(k))_{k \leq 7f^2\ell_n^{1/3}}$. By a small computation, since $f = o(n^{1/3})$, we have by Equation \eqref{eqf122}:
\begin{equation}
\begin{aligned}
\mathbb{E}[\textnormal{ht}(\mathbb{T}^{\top}(1))] &= \sum_{r = 0}^{\infty}\mathbb{P}(\textnormal{ht}(\mathbb{T}^{\top}(1)) \geq r) \\
&\leq 1+\sum_{r = 1}^{\infty}\frac{A}{r}\exp\left(\frac{-fr}{2\ell_n^{1/3}}\right)\\
&\leq 1+\sum_{r = 1}^{\frac{\ell_n^{1/3}}{f}}\frac{A}{r}
+\sum_{r = \frac{\ell_n^{1/3}}{f}+1}^{\infty}\frac{Af}{\ell_n^{1/3}}\exp\left(\frac{-fr}{2\ell_n^{1/3}}\right)\\
&\leq A'\log\left(\frac{\ell_n^{1/3}}{f}\right)+\frac{A'f}{\ell_n^{1/3}\left(1-\exp\left(\frac{-f}{2\ell_n^{1/3}}\right)\right)} \\
&= O\left(\log\left(\frac{n^{1/3}}{f}\right)\right)
\end{aligned}
\end{equation}
and similarily
\begin{equation}
\begin{aligned}
\mathbb{E}[\textnormal{ht}(\mathbb{T}^{\top}(1))^2] &= \sum_{r = 0}^{\infty}r^2\mathbb{P}(\textnormal{ht}(\mathbb{T}^{\top}(1))=r)\\
&\leq \sum_{r = 0}^{\infty}r(r+1)\mathbb{P}(\textnormal{ht}(\mathbb{T}^{\top}(1))=r) \\
&= \sum_{r = 0}^{\infty}2\mathbb{P}(\textnormal{ht}(\mathbb{T}^{\top}(1))=r)\left(\sum_{j=0}^r j\right) \\
&= 2\sum_{j = 0}^{\infty}j\left(\sum_{r=j}^\infty \mathbb{P}(\textnormal{ht}(\mathbb{T}^{\top}(1))=r)\right) \\
&= 2\sum_{j = 0}^{\infty}j\mathbb{P}(\textnormal{ht}(\mathbb{T}^{\top}(1))\geq j) \\
&=O\left(\frac{n^{1/3}}{f}\right).
\end{aligned}
\end{equation}
Hence, by Bernstein's inequality (\cite{B24}) for $A' >0$ large enough, there exists a large constant $A >0$ such that:
\begin{equation}
\begin{aligned}
\label{eql17}
&\mathbb{P}\left(\sum_{k= 1}^{R}\text{ht}(\mathbb{T}^{\top}(k))\mathbbm{1}(\tilde{\mathcal{B}}) \geq A'f^2\ell_n^{1/3}\log\left(\frac{n^{1/3}}{f}\right)\right) \\
&\leq \mathbb{P}\left(\sum_{k= 1}^{7f^2\ell_n^{1/3}}\text{ht}(\mathbb{T}^{\top}(k)))\mathbbm{1}(\text{ht}(\mathbb{T}^{\top}(k)) \leq \ell_n^{1/3}) \geq A'f^2\ell_n^{1/3}\log\left(\frac{n^{1/3}}{f}\right)\right)\\
&\leq A\exp\left(\frac{-\log\left(\frac{n^{1/3}}{f}\right)f^2}{A}\right).
\end{aligned}
\end{equation}
Suppose now that the height of the exploration tree of any connected component discovered before time $t$ is at most $A'f^2\ell_n^{1/3}\log\left(\frac{n^{1/3}}{f}\right)$, with $A'$ the constant of Equation \eqref{eql17}. Denote that event by $\mathcal{S}$. In that case, the number of cuts (denoted by $\mathcal{P}$) in one of the longuest paths in the exploration tree  of $H(\textbf{W},p_f)$ verifies by Bernstein's inequality (\cite{B24}):
\begin{equation}
\label{eql5}
\mathbb{P}\left(\mathcal{P}\mathbbm{1}(\mathcal{S}) \geq 5A'f^3\log\left(\frac{n^{1/3}}{f}\right)\right) \leq A''\exp\left(\frac{-\log\left(\frac{n^{1/3}}{f}\right)f^2}{A''}\right),
\end{equation}
We have for  $n$ large enough, for any $7f^2\ell_n^{1/3} \geq u \geq 1$: 
$$
\binom{7f^2\ell_n^{1/3}}{u} \leq \frac{(7f^2\ell_n^{1/3})^u}{u!} \leq \left(\frac{7f^2\ell_n^{1/3}e}{u}\right)^u
$$
With this we can bound the height of the $u$'th highest tree, denoted by $\mathbb{T}^{\top}_u$. By Equation \eqref{eqf122} there exist some large constants $A >0$ and $A' > 0$ such that:
\begin{equation}
\begin{aligned}
\label{eql6}
\mathbb{P}\left(\text{ht}(\mathbb{T}^{\top}_u)\mathbbm{1}(\tilde{\mathcal{B}}) \geq \frac{\ell_n^{1/3}}{\sqrt{u}}\right) &\leq \left(\frac{A\sqrt{u}}{\ell_n^{1/3}}\exp\left(\frac{-f}{2\sqrt{u}}\right)\right)^u\binom{7f^2\ell_n^{1/3}}{u} \\
&\leq \left(\frac{A\sqrt{u}}{\ell_n^{1/3}}\exp\left(\frac{-f}{2\sqrt{u}}\right)\right)^u\left(\frac{7ef^2\ell_n^{1/3}}{u}\right)^u \\
&\leq \left(\frac{7Aef^2}{\sqrt{u}}\exp\left(\frac{-f}{2\sqrt{u}}\right)\right)^u \\
&\leq A'\exp\left(\frac{-\sqrt{u}f-u\log{\left(\frac{u}{f}\right)}}{A'}\right).
\end{aligned}
\end{equation}
Let $m= 5A'f^3\log\left(\frac{n^{1/3}}{f}\right)$. If $\mathcal{P} \leq m$, then the height of the largest component is smaller than the sum of heights of the $m$ highest Galton-Watson trees $(\mathbb{T}^{\top}_u)_{u\leq m}$. By the union bound, using Equations \eqref{eql4},\eqref{eql17}, \eqref{eql5}, and \eqref{eql6}, there exist $A >0$ and $A'>0$ large enough such that:
\begin{equation}
\begin{aligned}
&\mathbb{P}\left(\sum_{u=1}^{m}\text{ht}(\mathbb{T}^{\top}_u) \geq \ell_n^{1/3}\left(\sum_{u = 1}^{ m}\frac{1}{\sqrt{u}}\right)\right) \\
&\leq A\exp\left(\frac{-f}{A}\right) + A\exp\left(\frac{-\log\left(\frac{n^{1/3}}{f}\right)f^2}{A}\right) + \sum_{u=1}^{m}A\exp\left(\frac{-\sqrt{u}f-u\log{\left(\frac{u}{f}\right)}}{A}\right) \\
&\leq A'\exp\left(\frac{-f}{A'}\right) + A\exp\left(\frac{-\log\left(\frac{n^{1/3}}{f}\right)f^2}{A}\right).
\end{aligned}
\end{equation}
Moreover, by comparison with an integral:
$$
\sum_{u = 1}^{ m}\frac{1}{\sqrt{u}} \leq 2\sqrt{m}.
$$
This shows that the height of the largest component is in fact smaller than $2\sqrt{m}\ell_n^{1/3}$ w.h.p. We can repeat the same argument as before to show that the number of cuts $\mathcal{P}$ in the largest path verifies:
$$
\mathbb{P}\left(\mathcal{P} \geq 9f\sqrt{m}\right) \leq A'\exp\left(\frac{-f}{A'}\right)+A\exp\left(\frac{-\log\left(\frac{n^{1/3}}{f}\right)f^2}{A}\right) + A\exp\left(\frac{-\sqrt{m}}{A}\right).
$$
By repeating the same arguments recursively, we get for any $l \geq 0$:
\begin{equation}
\mathbb{P}\left(\text{ht}(H(\textbf{W},p_f)) \geq \ell_n^{1/3}U_{l+1}\right) \leq A\exp\left(\frac{-f}{A}\right) + \sum_{u=0}^t A\exp\left(\frac{-U_l}{A}\right),
\end{equation}
where:
$$
U_0 = m,
$$
and:
$$
U_{l+1} = 9f\sqrt{U_l}.
$$
The sequence $(U_l)_{l \geq 0}$ converges to $81f^2$. If we define $R_l = U_l/(81f^2)$, then:
$$
R_{l+1} = \sqrt{R_l}.
$$
It is clear that for $l_0 = \lceil\log(\ln(R_0))\rceil$ we have $1 < R_{l_0} \leq e$. Hence:
\begin{equation}
\begin{aligned}
\mathbb{P}\left(\text{ht}(H(\textbf{W},p_f)) \geq \ell_n^{1/3}81ef^2\right) &\leq A\exp\left(\frac{-f}{A}\right) + \sum_{l=0}^{l_0} A\exp\left(\frac{-U_l}{A}\right) \\
&\leq A\exp\left(\frac{-f}{A}\right) + \sum_{l=0}^{l_0} A\exp\left(\frac{-81f^2R_{l_0}^{2^l}}{A}\right) \\
&\leq A\exp\left(\frac{-f}{A}\right) + \sum_{l=0}^{l_0} A'\exp\left(\frac{-81f^2(l+1)R_{l_0}}{A'}\right) \\
&\leq A\exp\left(\frac{-f}{A}\right) + A''\exp\left(\frac{-f^2}{A''}\right),
\end{aligned}
\end{equation}
which finishes the proof.
\end{proof}
With a little more work, we could show that the bound given above is in fact tight, in the sense that for any $\epsilon < 1$ the height of the largest component will be larger than $\frac{f^{1+\epsilon}\ell_n^{1/3}}{A}$ w.h.p. Finally by Theorem \ref{principal_a}, Lemma \ref{excess_height} and Theorem \ref{Theoo18} we get:
\begin{Theorem}
\label{theo19}
There exist constants $A' > 0$ and $A > 0$ such that:
$$
\mathbb{P}(\textnormal{diam}(H(\textbf{W},p_f)) \geq A'f^5\ell_n^{1/3}) \leq  A\exp\left(\frac{-f}{A}\right).
$$
\end{Theorem}

\section{Proofs of Theorems \ref{stats}, \ref{stats_2}, \ref{well-behaved} and \ref{critical-diameter}}
The different statements of Theorem \ref{well-behaved} were all proved in the previous sections. The statements about the weights are shown in Theorem \ref{principal_a}, and the statement about the longest path of small components is a direct corollary of Theorem \ref{them15}. \par
Now we finish  the proof of Theorem \ref{critical-diameter}.
recall that $(p_{f(i)})_{i \geq 0}$ is a sequence such that $f(0) = F$ the large constant defined at the end of Subsection $2.1$, and $f(i+1) = \frac{3}{2}f(i)$ for any $i \geq 0$.
We stop at $f(t_n)$ the smallest element larger than $f'_n = \frac{\ell_n^{1/3}}{\log(n)}$. For $0 \leq i \leq t_n$ we define the following events, $A >0$ is a large constant: 
\begin{enumerate}
    \item $E_1(i)$ is the event where $H(\textbf{W},p_{f(i)})$ has size between $\frac{3f(i)\ell_n^{2/3}}{2C}$ and $\frac{5f(i)\ell_n^{2/3}}{2C}$. 
    \item $E_2(i)$ is the event where the longest path of $H(\textbf{W},p_{f(i)})$ has length at most $Af(i)^5\ell_n^{1/3}$.
    \item $E_3(i)$ is the event where every connected component of $G(\mathcal{V} \setminus \tilde{\mathcal{V}}_i, p_{f(i+1)})$ has size at most $\frac{A\ell_n^{2/3}}{f(i)^{3/4}}$,  and longest path at most $\frac{A\ell_n^{1/3}}{f(i)^{1/4}}$.
\end{enumerate}

Remark that $j \leq 2\log(n)$. Let $r$ be the smallest value such that $E_1(i)$, $E_2(i)$, and $E_3(i)$ hold for every $r \leq l \leq j$. By Lemma \ref{croiss} and a simple computation:
$$
\text{diam}(H(\textbf{W},p_{f(j)})) \leq \text{diam}(H(\textbf{W},p_{f(r)})) + \frac{A'\ell_n^{1/3}}{f(r)^{1/5}} \leq A''f(r)^5\ell_n^{1/3}.
$$
If $r=i+1$ then one of the events $E_1(i)$, $E_2(i)$, or $E_3(i)$ does not hold, because if not we would have $r \leq i$. By Theorems \ref{principal_a}, \ref{Theo11}, \ref{Theo10}, \ref{them15} and Theorem \ref{theo19}, the probability of one of those events not happening is at most:
$$
\mathbb{P}(r = i+1) \leq A\exp\left(\frac{-f(i)^{1/8}}{A}\right).
$$
Combining the two inequalities above yields:
\begin{equation*}
\begin{aligned}
\mathbb{E}\left[\text{diam}(H(\textbf{W},p_{f(j)}))\right]  &\leq \sum_{i=1}^{j-1}A''f(i+1)^5\ell_n^{1/3}\mathbb{P}(r = i+1) \\
&\leq \sum_{i=1}^{j-1}A''f(i+1)^5\ell_n^{1/3}A\exp\left(\frac{-f(i)^{1/8}}{A}\right) \\
&= O(n^{1/3}).
\end{aligned}
\end{equation*}
The lower bound was provided in Theorem \ref{lower_bound}.
We move now to Theorem \ref{stats}.
\begin{proof}[Proof of Theorem \ref{stats}]
Since the diameter is an upper bound of typical distances, by Theorem $1$ we already have:
$$
\mathbb{E}[d(U_{f_n},V_{f_n})] = O(n^{1/3}).
$$
We need to show now that there exist some constants $\epsilon > 0$ and $\epsilon' > 0$ such that typical distances in the largest tree in $\mathcal{T}(\textbf{W},\tilde{p}_n)$ are larger than $\epsilon n^{1/3}$ with probability at least $\epsilon'$. Denote that tree by $B(\tilde{p}_n)$. And more generally for any $p > 0$ denote the largest tree in $\mathcal{T}(w,p)$  by $B(p)$. Let $f > 0$ be a large constant, and $p_f=\frac{1}{\ell_n}+\frac{f}{\ell_n^{4/3}}$. For $n$ large enough $p_f < \tilde{p}_n$. For a constant $\epsilon > 0$, consider the following events: 
\begin{itemize}
    \item $\mathcal{A}_1$ is the event where  $B(p_f)$
is a sub-tree of $B(\tilde{p}_n)$, and the size and weight of $B(p_f)$ are smaller than $3f\ell_n^{2/3}$ for every $n$ large enough.

\item $\mathcal{A}_2$ is the event where the size and weight of $B(\tilde{p}_n)$ is larger than $\epsilon n$ for every $n$ large enough.
\item $\mathcal{A}_3$ is the event where there exist two sub-trees of $B(p_f)$, that we denote by $B_1(p_f)$ and $B_2(p_f)$, such that the minimal distance in $B(p_f)$ between a node in $B_1(p_f)$ and a node in $B_2(p_f)$ is larger than $\epsilon n^{1/3}$. Moreover, the sizes and weights of $B_1(p_f)$ and $B_2(p_f)$ are larger than $\epsilon n^{2/3}$ for every $n$ large enough.
\end{itemize}
Then by Theorem \ref{principal_a}, for any $\epsilon_1 >0$, if $f$ is large enough then $\mathcal{A}_1$ holds with probability larger than $1-\epsilon_1$. Moreover, by Theorem \ref{phase_trans}, there exists $\epsilon > 0$ such that $A_2$ holds with high probability.
And also, \cite{br20} show the convergence of the largest component of $G(\textbf{W},p_f)$ with graph distances scaled by $n^{1/3}$ and sizes scaled by $n^{2/3}$ in the Gromov-Hausdorff-Prokhorov topology to a continuous compact connected graph. We refer to their article for a detailed analysis of this general result. Here we just emphasize the fact that their limiting object is a non-trivial compact metric space which containse two subspaces of mass $\zeta > 0$ and at distance at least $\zeta$ with positive probability $\zeta' > 0$. This result directly implies the existence of $\epsilon_3 > 0$ such that $A_3$ holds with probability at least $\epsilon_3$. \par
Let $ \mathcal{A}(\epsilon) = \mathcal{A}_1 \cap \mathcal{A}_2 \cap \mathcal{A}_3$. By the discussion above, there exists two constants $\epsilon >0$ and $\epsilon_4 > 0$ such that, for a constant $f > 0$ large enough, $\mathcal{A}(\epsilon)$ holds with probability at least $\epsilon_4$. This shows that there exists some $\epsilon_5 > 0$ such that, by denoting the total weight of a graph $G$ by $W(G)$, we have \footnote{Here we just set the weights and sizes to $0$ if $B_1(p_f)$ or $B_2(p_f)$ do not exist}: 
\begin{equation}
\label{eqfin1}
\mathbb{E}\left[W\left(B_1(p_f)\right)\right] \geq \epsilon_5\mathbb{E}\left[W\left(B(p_f)\right)\right] .
\end{equation}
For $\tilde{p}_n \geq p \geq p_f$, let  $B_1(p)$ be the sub-tree of  $B(\tilde{p}_n)$ obtained by applying Kruskal's algorithm to   $B_1(p_f)$  between $p_f$ and $p$. And define $B_2(p)$ similarly. Clearly the distance between $B_1(\tilde{p}_n)$ and $B_2(\tilde{p}_n)$ is larger than $\epsilon_3n^{1/3}$. It is thus sufficient to show that the sizes of $B_1(\tilde{p}_n)$ and $B_2(\tilde{p}_n)$ are on expectation proportional to $n$. By symmetry it is sufficient to show this for $B_1(\tilde{p}_n)$. \par
Let $(r_1,r_2,r_3,...)$ be the sequence of increasing values between $p_f$ and $\tilde{p}_n$ at which a new edge is added by Kruskal's algorithm to $(B(r_i^{-}))_{i \geq 1}$. At each value $r_i$, an edge is created between $B(r_i^{-})$ and the rest of $\mathcal{T}(\textbf{W},r_i^{-})$. And by the properties of exponential random variables, the probability that such an edge gets connected to $B_1(r_i^{-})$ is equal to 
$$
\frac{W(B_1(r_i^{-}))}{W(B(r_i^{-}))},
$$
Hence the process of weights $$(W(B_1(r_i^{-})),W(B(r_i))-W(B_1(r_i^{-})))_{i \geq 1},$$ corresponds to a Polya Urn with a ball with of random weight is added at each step, it is easy to check by induction that if at step $i$, there exists $\epsilon > 0$ such that:
$$
\mathbb{E}[W(B_1(r_i^-))] = \epsilon \mathbb{E}[W(B(r_i^-))],
$$
Then at step $i+1$, a connected component $C(i)$ is added and we get:
\begin{equation*}
\begin{aligned}
\mathbb{E}[W(B_1(r_{i+1}^-))] &= \epsilon \mathbb{E}[W(B(r_i^-))] + \epsilon \mathbb{E}[W(C(i))] \\
&= \epsilon \mathbb{E}[W(B(r_{i+1}^-)],
\end{aligned}
\end{equation*}
with the same $\epsilon$ of step $i$.
We know that (for instance by Theorem \ref{phase_trans}), there exists $\epsilon_8 >0$ such that for $n$ large enough
$$
\mathbb{E}\left[W\left(B\left(\tilde{p}_n\right)\right)\right]\geq \epsilon_8n.
$$
Given that the event $\mathcal{A}$ holds with positive probability, this, alongside Equation \ref{eqfin1}, shows that there exists $\epsilon_6 > 0$ such that: 
\begin{equation}
\label{eqfin}
\mathbb{E}\left[W\left(B_1\left(\tilde{p}_n\right)\right)\right] \geq \epsilon_6n.
\end{equation}
And the same also holds for $B_2$. Finally, we use this lower bound on the weight in order to get a lower bound on what we actually need, the size. \par
Recall that the node weights depend implicitly on $n$. Suppose that there exists an $ N > 0$ and subsets $(\mathcal{S}_n)_{n \geq N}$ of the nodes such that:
$$
\sum_{k \in \mathcal{S}_n}w_k \geq \epsilon_6 n,
$$
for every $n \geq N$, but the sizes of those sets verify  $|\mathcal{S}_n| = o(n)$. Let $w_{V_n}$ be a uniformly chosen node from $(w_1,w_2,...,w_n)$.  We thus have:
$$
\mathbb{E}[\mathbbm{1}(V_n \in S_n)w_{V_n}] \geq \epsilon_7.
$$
Conditions $i$ and $iv$ in Conditions \ref{cond_nodess} imply by Theorem $3.6$ in \citet*{BL13} that the $(w_{V_n})_{n \geq N}$ are uniformly integrable.
However, by definition of $(\mathcal{S}_n)_{n \geq N}$ we also have $\mathbb{P}(V_n \in S_n) = o(1)$. This is in contradiction with the definition of uniform integrability . Hence no such sets exist, and Equation \ref{eqfin} implies the existence of $\epsilon_7$ such that:
\begin{equation*}
\mathbb{E}\left[\middle|B_1\left(\tilde{p}_n\right)\middle|\right] \geq \epsilon_7n,
\end{equation*}
for any $n$ large enough. We have thus shown that there exist two sub-trees, $B_1\left(\tilde{p}_n\right)$ and $B_1\left(\tilde{p}_n\right)$, of $B\left(\tilde{p}_n\right)$ of sizes proportional to $n$ and such that the distance between the two sub-trees is larger than $\epsilon_7n$ on-expectation. This finishes the proof. \end{proof}

Finally we prove the same result for the trees related to statistical physics. Meaning we take the minimum spanning tree of the largest component of $G(\textbf{W},\tilde{p}_n)$ with i.i.d capacities on its edges.

\begin{proof}[Proof of Theorem \ref{stats_2}]
Under Conditions \ref{cond_nodess}, and for $n$ large enough, we can also construct that minimal spanning tree by using an edge deletion algorithm with exponentially distributed edge capacities. Fix $c > 1$ and to each  non-oriented edge $\{i,j\}$, $i \neq j$, associate the random capacity $E'_{\{i,j\}}$, which is an exponential random variable of rate $1$. The capacities are then used to create a sequence of graphs. First keep each edge independently with probability $p_{\{i,j\}} = \frac{w_iw_jc}{\ell_n} = o(n^{-1/3})$. Once this operation is done we obtain a graph $G'(\textbf{W},+\infty)$. We can then construct an increasing sequence of graphs $(G'(\textbf{W},p))_{p \geq 0}$ for inclusion by letting $G'(\textbf{W},p)$ be the subgraph of $G'(\textbf{W},+\infty)$ with edge capacities verifying:
$$\left\{ \{i,j\} |  E'_{\{i,j\}} \leq p \right\}.$$ 
Under Conditions \ref{cond_nodess}:
$$
1 -\exp\left(\frac{-w_iw_jc}{\ell_n}\right) = p_{\{i,j\}}\left(1+O(p_{\{i,j\}})\right).
$$
Moreover, if $\{i,j\} \neq \{i',j'\}$ then the edge capacities  $E'_{\{i,j\}}$ and $E'_{\{i',j'\}}$ are independent. This implies that for any $\epsilon > 0$ and for $n$ large enough $G'(\textbf{W},+\infty)$ can be coupled edge by edge with two graphs $G(\textbf{W},\frac{c+\epsilon \ell_n^{-1/3}}{\ell_n})$ and $G(\textbf{W},\frac{c-\epsilon \ell_n^{-1/3}}{\ell_n})$ in such a way that: 
$$ 
G\left(\textbf{W},\frac{c-\epsilon \ell_n^{-1/3}}{\ell_n}\right) \subset G'(\textbf{W},+\infty) \subset G\left(\textbf{W},\frac{c+\epsilon \ell_n^{-1/3}}{\ell_n}\right).
$$
We say that $G'(\textbf{W},+\infty)$ is asymptotically equivalent to $G(\textbf{W},\tilde{p}_n)$. Generally, for every $p \geq 0$, $G'(\textbf{W},p)$ is asymptotically equivalent to $G(\textbf{W},(1-e^{-p})\tilde{p}_n)$. 
Now for each $p$, consider the forest $\mathcal{T}'(\textbf{W},p)$ constructed by the edge deletion algorithm: 
Sort the edges $(\{i,j\})_{i \leq n, j\leq n}$ of $G'(\textbf{W},p)$ by decreasing order of their capacities $(E'_{\{i,j\}})_{i \leq n, j\leq n}$. Then starting from the first edge, only keep an edge if removing it would disconnect a connected component.
By construction, for any $p \leq \infty$, $\mathcal{T}'(\textbf{W},p)$ is the  minimum spanning forest of $G'(\textbf{W},p)$ with respect to the capacities $(E'_{\{i,j\}})_{i \leq n, j\leq n}$. Hence, $(\mathcal{T}'(\textbf{W},p))_{p\geq 0}$ is an increasing process of spanning trees for $(G'(\textbf{W},p)_{p\geq 0}$. Hence, if we define $\mathcal{A}'(\epsilon)$ similarly to $\mathcal{A}(\epsilon)$ but for the forests $(\mathcal{T}'(\textbf{W},p))_{p \geq 0}$, it is sufficient to prove that there exist $\epsilon >0$ and $\epsilon'>0$ such that $\mathcal{A}'(\epsilon)$ holds with probability at least $\epsilon'$. This is true because $\mathcal{A}(\epsilon)$ holds with probability at least $\epsilon'$ for some $\epsilon' > 0$ and $\epsilon > 0$ and because of the asymptotic equivalence relation between $G(\textbf{W},(1-e^{-p})\tilde{p}_n)$ and $G'(\textbf{W},p)$ for every $p \geq 0$, and between $G'(\textbf{W},+\infty)$ and $G(\textbf{W},\tilde{p}_n)$.
\end{proof}
\newpage

\renewcommand\refname{Bibliography}

\bibliographystyle{abbrvnat} % or try abbrvnat or unsrtnat
\bibliography{bibli} % refers to example.bib

\end{document}